\UseRawInputEncoding
\documentclass[11pt, reqno, oneside]{amsart}
\usepackage[T1]{fontenc}
\usepackage[utf8]{inputenc}
\usepackage[british]{babel}
\usepackage{amsmath,amssymb,amsthm,mathtools}
\usepackage{stmaryrd}
\usepackage{geometry}
\usepackage{hyperref}
\usepackage{enumitem}
\hypersetup{colorlinks=true,linkcolor=blue,citecolor=blue,urlcolor=blue}
\usepackage{caption, subcaption}
\usepackage{comment}

\usepackage{cleveref}

\usepackage{tikz-cd}

\tikzset{
curarrow/.style={
rounded corners=10pt,
execute at begin to={every node/.style={fill=red}},
to path={-- ([xshift=50pt]\tikztostart.center)
  |- (#1) node[fill=white] {$\scriptstyle \delta$}
  -| ([xshift=-40pt]\tikztotarget.center)
  -- (\tikztotarget)}
  }
}

\tikzset{
curvararrow/.style={
rounded corners=10pt,
execute at begin to={every node/.style={fill=red}},
to path={-- ([xshift=70pt]\tikztostart.center)
  |- (#1) node[fill=white] {$\scriptstyle \delta$}
  -| ([xshift=-40pt]\tikztotarget.center)
  -- (\tikztotarget)}
  }
}

\usepackage{adjustbox}

\providecommand\given{}
\newcommand\SetSymbol[1][]{%
\nonscript\:#1\vert
\allowbreak
\nonscript\:
\mathopen{}}
\DeclarePairedDelimiterX\Set[1]\{\}{%
\renewcommand\given{\SetSymbol[\delimsize]}
#1
}

\title{Equivariant cosymplectic geometry}
\author{Eva Miranda}
\address{Eva Miranda,
Laboratory of Geometry and Dynamical Systems \& SYMCREA, Department of Mathematics, EPSEB, Universitat Polit\`{e}cnica de Catalunya-IMTech
in Barcelona and
\\ CRM Centre de Recerca Matem\`{a}tica, Campus de Bellaterra
Edifici C, 08193 Bellaterra, Barcelona.
 }
\thanks{Corresponding author: Eva Miranda. Email: \texttt{eva.miranda@upc.edu}}
\thanks{Eva Miranda is funded by the Catalan Institution for Research and Advanced Studies via an ICREA Academia Prize 2021 and by the Alexander von Humboldt Foundation via a Friedrich Wilhelm Bessel Research Award.}
\thanks{Both authors are supported by the Spanish State
Research Agency, through the Severo Ochoa and Mar\'{\i}a de Maeztu Program for Centers and Units
of Excellence in R\&D (project CEX2020-001084-M), and under grant reference PID2023-146936NB-I00 funded by MICIU/AEI/10.13039/501100011033 and, by ERDF/EU.} 
\author{Pablo Nicolás}
\address{Pablo Nicolás,
Centre de Recerca Matemàtica, CRM \& Laboratory of Geometry and Dynamical Systems and SYMCREA research unit, Department of Mathematics, Universitat Polit\`{e}cnica de Catalunya, Barcelona}
\thanks{Pablo Nicolás is supported under a MDM-FPI contract with reference code PRE2022-102974.}
\date{\today}

\usepackage{thmtools}

\declaretheoremstyle[
  spaceabove=1em, spacebelow=1em,
  headfont=\bfseries,
  notefont=\bfseries, notebraces={(}{)},
  bodyfont=\itshape,
  postheadspace=0.5em,
]{thmlike}

\declaretheoremstyle[
  spaceabove=1em, spacebelow=1em,
  headfont=\normalfont,
  notefont=\normalfont, notebraces={(}{)},
  bodyfont=\normalfont,
  postheadspace=1em,
  qed=$\blacksquare$
]{prooflike}

\declaretheoremstyle[
  spaceabove=1em, spacebelow=1em,
  headfont=\bfseries,
  notefont=\bfseries, notebraces={(}{)},
  bodyfont=\normalfont,
  postheadspace=0.5em,
]{deflike}

\declaretheorem[style=thmlike, numberwithin=section]{theorem}
\declaretheorem[style=thmlike, sibling=theorem]{corollary}
\declaretheorem[style=deflike, sibling=theorem]{definition}
\declaretheorem[style=thmlike, sibling=theorem]{lemma}
\declaretheorem[style=thmlike, sibling=theorem]{proposition}

\declaretheorem[style=deflike, sibling=theorem]{remark}

\declaretheorem[style=thmlike, name=Main Theorem]{maintheorem}

\declaretheorem[style=deflike, sibling=theorem, numbered=no, name=Definition]{definition*}
\declaretheorem[style=deflike, sibling=theorem, numbered=no, name=Example]{example*}

\newcommand{\F}{\mathcal F}

\newcommand*\diff{\mathop{}\!\mathrm{d}}

\makeatletter

\DeclareRobustCommand\cotimes{\mathbin{\widehat{\otimes}\relax}}

\newcommand{\c@otimes}[2]{%
  \vbox{
    \ialign{##\cr
      \hidewidth$\m@th#1{}_\frown$\kern-\scriptspace\hidewidth\cr
      \noalign{\nointerlineskip\kern-1pt}
      $\m@th#1\otimes$\cr
    }%
  }%
}
\makeatother

\begin{document}

\begin{abstract}
Cosymplectic manifolds provide a natural geometric framework for codimension-one symplectic foliations and arise throughout geometry and mathematical physics. We develop an equivariant cohomological theory for cosymplectic manifolds, studying the role of symmetry in their topology and their connections to Poisson geometry.

Reinterpreting the obstruction theory of Guillemin, Miranda, and Pires~\cite{guillemin_codimension_2011}, we introduce equivariant obstruction classes for group actions preserving the foliation and characterize the existence of invariant cosymplectic structures. We further show that the vanishing of the first obstruction class is equivalent to equivariant unimodularity of the associated Poisson structure, extending a classical criterion to the equivariant setting via Ginzburg's framework for equivariant Poisson cohomology~\cite{ginzburg_equivariant_1999}.

For compact cosymplectic manifolds fibering over $\mathbb{S}^1$, we construct equivariant versions of de~Rham, foliated, and Poisson cohomologies, establishing formality results in each case via an equivariant Wang sequence. The key input is Kirwan's formality theorem for Hamiltonian actions on the symplectic fiber~\cite{kirwan}, which, through the equivariant Wang sequence, yields a complete and computable description of all three equivariant cohomology theories in terms of the monodromy action on the fiber.

\end{abstract}

\maketitle

\section{Introduction}

Symmetry has long served as a guiding principle in geometry.
From Hamiltonian group actions in symplectic geometry to moment maps in geometric representation theory, equivariant structures provide a language through which topology, dynamics, and geometry become deeply intertwined.
In odd dimensions, however, the landscape is subtler.
Alongside contact geometry, there is another natural counterpart to symplectic geometry: cosymplectic geometry, a theory whose rigidity and global character conceal a rich cohomological structure.

A cosymplectic manifold is a smooth manifold of dimension $(2n+1)$ endowed with a pair of closed differential forms $(\alpha,\beta)$ satisfying $\alpha \wedge \beta^n \neq 0$. The form $\alpha$ defines a codimension-one foliation $\mathcal{F}=\ker \alpha$, while $\beta$ restricts to a symplectic form on every leaf.
Thus cosymplectic geometry occupies a distinguished position at the interface between foliation theory and symplectic geometry: it may be viewed both as a symplectic foliation with transverse rigidity and as an odd-dimensional shadow of symplectic geometry.

Not every codimension-one symplectic foliation arises from a global cosymplectic structure.
The problem of determining when such a structure exists was resolved by Guillemin, Miranda, and Pires in \cite{guillemin_codimension_2011}, who introduced two obstruction classes in foliated cohomology whose vanishing precisely characterizes the existence of a compatible cosymplectic structure.
Their work revealed that the passage from leafwise symplectic data to a global geometric structure is governed by subtle cohomological phenomena.

Cosymplectic geometry also appears naturally within Poisson geometry.
The critical hypersurface of a $b$-symplectic manifold inherits a canonical cosymplectic structure \cite{guillemin_codimension_2011,GMP14}, making cosymplectic manifolds the geometric avatars of degeneracy loci in singular symplectic geometry.
They arise naturally in the study of Poisson manifolds with controlled singularities, modular vector fields, and desingularization phenomena, suggesting a broader framework in which foliations, symmetry, and cohomological invariants coexist.

The purpose of this article is to develop an equivariant counterpart of this theory.
Given a compact Lie group acting on a manifold and preserving a codimension-one symplectic foliation, we ask whether the foliation admits a compatible invariant cosymplectic structure, and to what extent the classical obstruction theory survives in the equivariant setting.
Our approach combines equivariant foliated cohomology, the Cartan model of equivariant differential forms, and the geometry of leafwise moment maps.

\subsection*{Main results}

The first main result extends the obstruction classes of Guillemin, Miranda, and Pires \cite{guillemin_codimension_2011} to the equivariant setting.

\begin{maintheorem}[Equivariant obstruction classes]
Let $(M,\mathcal{F},\beta)$ be a coorientable codimension-one symplectic foliation preserved by a compact connected Lie group $G$.
There exist natural equivariant obstruction classes
\[
    c_{\mathcal{F}}^G \in \mathrm{H}_G^1(\mathcal{F}),
    \qquad
    \sigma_{\mathcal{F}}^G \in \mathrm{H}_G^2(\mathcal{F}),
\]
whose vanishing is equivalent to the existence of a $G$-equivariant cosymplectic structure inducing the foliation.\footnote{
    Observe that since the vanishing of the first obstruction class $c_{\mathcal{F}}^G$ is a necessary condition, \Cref{cor:second-obstruction-class} applies and we may regard the second obstruction class $\sigma_{\mathcal{F}}^G$ as an element of $\mathrm{H}^2(\mathcal{F})$.
}
\end{maintheorem}

These classes detect whether the defining forms of the cosymplectic structure may be chosen equivariantly closed in the Cartan model, and they encode the obstruction to globally assembling local Hamiltonian data into a collective moment map.
In the non-equivariant case we obtain a more intrinsic and cleaner cohomological description of the obstruction theory associated to the foliation in \cite{guillemin_codimension_2011}.
The first obstruction class is moreover related to the Poisson geometry of $M$: using the equivariant Poisson cohomology framework of Ginzburg \cite{ginzburg_equivariant_1999}, we prove the following equivariant analogue of a classical criterion of Guillemin, Miranda, and Pires.

\begin{maintheorem}[Equivariant unimodularity]
Under the assumptions above, the Poisson manifold $(M,\Pi)$ is $G$-unimodular if and only if the first obstruction class
\[
    c_{\mathcal{F}}^G \in \mathrm{H}_G^1(\mathcal{F})
\]
vanishes.
\end{maintheorem}

The second part of the paper concerns the topology of compact cosymplectic manifolds.
By foundational results of Tischler \cite{tischler_fibering_1970} and Li \cite{li_topology_2008}, compact cosymplectic manifolds with compact leaves fiber over the circle and may be realized as symplectic mapping tori.
This global description permits explicit computations of de~Rham, foliated, and Poisson cohomology, both in the classical and equivariant settings.

For the class of cosymplectic manifolds fibering over $\mathbb{S}^1$, we develop equivariant analogues of the Wang long exact sequence for de~Rham, foliated, and Poisson cohomology.
These sequences express the corresponding equivariant cohomology groups in terms of the equivariant cohomology of the symplectic fiber and the induced monodromy action.
When the action on the fiber is Hamiltonian, Kirwan's formality theorem \cite{kirwan} further applies to the fiber, yielding an explicit monodromy-controlled description of the equivariant cohomology of the mapping torus.

\begin{maintheorem}[Equivariant Wang sequences and formality]
Let $(M,\alpha,\beta)$ be a compact cosymplectic manifold fibering over $\mathbb{S}^1$ with compact symplectic fiber $L$, and let $G$ be a compact connected Lie group preserving the cosymplectic structure.
Then the equivariant de~Rham, foliated, and Poisson cohomology theories of $M$ fit into natural Wang-type long exact sequences determined by the monodromy action on the corresponding equivariant cohomology of $L$.

If, in addition, the induced $G$-action on the symplectic fiber $L$ is Hamiltonian, then Kirwan formality applies to $L$.
In this case, the Wang sequences give an explicit description of the equivariant cohomology theories of $M$ in terms of the monodromy action on $\mathrm{H}^\bullet(L)$.
\end{maintheorem}

These results place cosymplectic geometry within the broader equivariant framework that has proved so fruitful in symplectic and Poisson geometry, and they suggest that many familiar phenomena from Hamiltonian geometry persist, in a suitably adapted form, in the realm of codimension-one symplectic foliations.
Further connections with reduction, localization phenomena, and convexity for cosymplectic manifolds will be studied in a companion article \cite{miranda_hamiltonian_2026}.

\subsection*{Structure of the paper}

\Cref{sec:preliminaries} reviews the background material needed throughout the article: foliations and foliated cohomology, Poisson geometry and the modular class, equivariant cohomology in the Cartan model, and Ginzburg's construction of equivariant Poisson cohomology.

\Cref{sec:equiv-coorient-codim1-foliations} studies equivariant, co-orientable, codimension-one foliations.
We establish equivariant analogues of the long exact sequences in foliated cohomology and prove that the relevant splittings may be chosen $G$-invariantly.

\Cref{sec:equivariant-obstruction-classes} introduces equivariant cosymplectic structures and establishes the equivariant obstruction theory.
We define the equivariant obstruction classes, prove that their vanishing characterizes the existence of invariant cosymplectic structures, and connect the first obstruction class to equivariant unimodularity of the associated Poisson structure.

\Cref{sec:cohomology-cosymplectic} computes the de~Rham, foliated, and Poisson cohomology groups of compact cosymplectic manifolds using their description as symplectic mapping tori.

\Cref{sec:equiv-cohomology-cosymplectic} develops the equivariant counterparts of these cohomology theories and establishes the corresponding Wang-type exact sequences and formality consequences.

\section{Preliminaries} \label{sec:preliminaries}

\subsection{Foliations and foliated cohomology}

A foliation provides a decomposition of a manifold into immersed submanifolds (the \emph{leaves}) that locally resemble parallel affine subspaces. One way to encode this information is in terms of the involutive distribution generated by tangent vector fields to all leaves.

\begin{definition}
    A \emph{(regular) foliation} of dimension $k$ on a smooth manifold $M$ is a rank-$k$ subbundle $\mathrm{T}\F \subset \mathrm{T}M$ which is \emph{involutive}, i.e., $[X,Y]\in\Gamma(\mathrm{T}\F)$ for all $X,Y\in\Gamma(\mathrm{T}\F)$. Its integral manifolds are called the \emph{leaves} of the foliation. We will usually denote the space of sections $\Gamma(\mathrm{T} \mathcal{F})$ by $\mathfrak{X}(\mathcal{F})$.
\end{definition}

One may alternatively define a regular foliation by its annihilator ideal, i.e., the ideal of smooth forms satisfying
\begin{equation}
    \mathcal{I}^k (\mathcal{F}) = \Set{ \omega \in \Omega^k(M) \given \omega(X_1, \ldots, X_k) = 0 \text{ for all } X_i \in \mathfrak{X}(\mathcal{F})  }.
\end{equation}

The involutivity of the distribution $\mathcal{F}$ is reflected in the characterization in terms of the characteristic ideal following Frobenius' theorem.

\begin{theorem}[Frobenius]
    Let $\mathcal{D} \subset \mathrm{T}M$ be a rank-$k$ subbundle. The following are equivalent:
    \begin{enumerate}
        \item $\mathcal{D}$ is integrable,
        \item $\mathcal{D}$ is involutive, and
        \item the ideal $\mathcal{I}^\bullet(\mathcal{D})\subset \Omega^\bullet(M)$ is differential, i.e., $\diff \mathcal{I}^\bullet (\mathcal{D}) \subset \mathcal{I}^{\bullet + 1}(\mathcal{D})$.
    \end{enumerate}
\end{theorem}

Let us denote the space of forms in the foliation by $\Omega^k(\mathcal{F}) \coloneq \Gamma \bigwedge^k \mathrm{T}^* \mathcal{F}$. The dual map $j \colon \mathrm{T}^* M \to \mathrm{T}^* \mathcal{F}$ to the inclusion $i \colon \mathrm{T} \mathcal{F} \to \mathrm{T} M$ induces a map at the level of sections $j \colon \Omega^\bullet(M) \to \Omega^\bullet(\mathcal{F})$. Since the characteristic ideal $\mathcal{I}^\bullet(\mathcal{F}) = \ker j$ is stable under the differential $\diff$, the exterior differential projects to an operator
\begin{equation}
    \diff_{\mathcal{F}} \colon \Omega^\bullet(\mathcal{F}) \longrightarrow \Omega^{\bullet  + 1}(\mathcal{F})
\end{equation}
called the \emph{foliated differential}. We have, by definition, that $\diff_{\mathcal{F}}^2 = 0$. Therefore, we have a cochain complex whose cohomology groups are denoted by $\mathrm{H}^\bullet (\mathcal{F})$ and called \emph{foliated cohomology groups}.

\subsection{Fundamentals of Poisson geometry} \label{ssec:poisson-geometry}

A Poisson manifold is a smooth manifold $M$ equipped with a bilinear bracket
\[
\{\cdot,\cdot\}\colon \mathcal{C}^\infty(M)\times \mathcal{C}^\infty(M)
\longrightarrow \mathcal{C}^\infty(M),
\]
which is a Lie bracket and satisfies the Leibniz rule in each argument.

Since $\{\cdot,\cdot\}$ is a derivation in each argument, it is induced by a unique bivector field
\[
    \Pi \in \mathfrak{X}^2(M).
\]
Equivalently, for all $f,g\in \mathcal{C}^\infty(M)$,
\begin{equation}
    \{f,g\} = \Pi(\diff f,\diff g).
\end{equation}
The Jacobi identity for $\{\cdot,\cdot\}$ is equivalent to
\[
    [\Pi,\Pi]=0,
\]
where $[\cdot,\cdot]$ denotes the Schouten--Nijenhuis bracket. The bivector $\Pi$ is called a \emph{Poisson structure}, and the pair $(M,\Pi)$ is called a \emph{Poisson manifold} \cite{weinstein_local_1983,cannas_da_silva_geometric_1999,dufour_normal_2011}.

Since the function $\{f, \cdot\} \colon \mathcal{C}^\infty(M) \to \mathcal{C}^\infty(M)$ is a derivation, it corresponds to a unique vector field $X_f \in \mathfrak{X}(M)$ called the \emph{Hamiltonian vector field} of $f$. This class of vector fields is fundamental in the study of the global structure of Poisson manifolds: the distribution $\mathcal{D}$ generated by all Hamiltonian vector fields is involutive and locally finitely generated, and hence integrates by the Stefan--Sussman theorem to a foliation $\mathcal{F}_\Pi$ called the \emph{symplectic foliation} of $M$. This foliation is singular, in general, and the restriction of $\Pi$ to each leaf induces a symplectic form. Therefore, Poisson manifolds decompose in symplectic pieces. 

Equivalently, one may recover Hamiltonian vector fields and the symplectic foliation $\mathcal{F}_\Pi$ in terms of the vector bundle morphism $\Pi^\sharp \colon \mathrm{T}^*M \to \mathrm{T}M$, defined as $\Pi^\sharp(\alpha) = \iota_\alpha \Pi$ and called the \emph{anchor map}. In this notation, we have $X_f = \Pi^\sharp(\diff f)$ and $\mathcal{F}_\Pi$ is the distribution corresponding to $\operatorname{im} \Pi^\sharp \subseteq \mathrm{T}M$.

The infinitesimal structure of a Poisson tensor is encoded by the operator
\[
    \diff_\Pi \coloneqq [\Pi,\cdot],
\]
acting on multivector fields $\mathfrak{X}^\bullet(M)$. Since $[\Pi,\Pi]=0$, one has $\diff_\Pi^2=0$. The cochain complex $(\mathfrak{X}^\bullet(M), \diff_\Pi)$ is called the \emph{Lichnerowicz complex}, and its corresponding cohomology groups $\mathrm{H}^\bullet_\Pi(M)$ are called \emph{Poisson cohomology groups}. These objects govern deformations and symmetries of $\Pi$.

Let us assume, for the sake of simplicity, that $(M,\Pi)$ is orientable and fix a volume form $\Theta$. The divergence operator $\operatorname{div}_\Theta \colon \mathfrak{X}(M) \to \mathcal{C}^\infty(M)$ applied to Hamiltonian vector fields induces a map $f \mapsto \operatorname{div}_\Theta X_f$ which turns out to be a derivation. Thus, it corresponds to a unique vector field $X^\Theta \in \mathfrak{X}(M)$ called the \emph{modular vector field} of $\Theta$. This vector field is, by definition, characterized by the relation
\begin{equation}
    \mathcal{L}_{X_f}\Theta = (\mathcal{L}_{X^\Theta} f) \,\Theta
\end{equation}
for all $f \in \mathcal{C}^\infty(M)$.

The modular vector field $X^\Theta$ is always a Poisson vector field and tangent to the symplectic leaves of $M$ at regular points. Moreover, for a different volume form $\Xi = f \Theta$ the corresponding modular vector fields are related by
\begin{equation} \label{eq:mod-vf-coboundary}
    X^\Xi = X^\Theta - X_{\log \lvert f \rvert}.
\end{equation}
Therefore, the cohomology class $[X^\Theta] \in \mathrm{H}_\Pi^1(M)$ is well-defined and independent of the choice of volume form $\Theta$. It is called the \emph{modular class} of $(M, \Pi)$ \cite{weinstein_modular_1997}, and its vanishing characterizes unimodular Poisson structures. Recall that a Poisson manifold is unimodular if and only if there exists a volume form $\Theta \in \Omega^n(M)$ which is invariant under all Hamiltonian vector fields.

A particularly interesting class of Poisson manifolds arises from regular symplectic foliations, i.e., a regular foliation $\mathcal{F}$ on $M$ equipped with a foliated two-form $\omega \in \Omega^2(\mathcal{F})$ which is $\diff_{\mathcal{F}}$-closed and non-degenerate. Under this assumption we may define the Hamiltonian vector field $X_f \in \mathfrak{X}(\mathcal{F})$ of a smooth function $f \in \mathcal{C}^\infty(M)$ as the unique solution to Hamilton's equations $\iota_{X_f} \omega = - \diff_{\mathcal{F}} f$, and a Poisson bracket as
\begin{equation}
    \{f, g\} = \omega(X_f, X_g).
\end{equation}
The corresponding Poisson bivector is simply given $\Pi = i\omega^\sharp(\omega)$, where $\omega^\sharp \colon \mathrm{T}^* \mathcal{F} \to \mathrm{T} \mathcal{F}$ is the inverse map to $\omega^\flat(X) = \iota_X \omega$ and $i \colon \mathrm{T} \mathcal{F} \to \mathrm{T} M$ is the natural inclusion of tangent vectors. We may also write $\Pi_{\mathcal{F}} \coloneqq \omega^\sharp(\omega) \in \mathfrak{X}^2(\mathcal{F})$ to denote the restriction of the bivector $\Pi$ to the complex of multivector fields tangent to the foliation $\mathcal{F}$. The anchor map of the Poisson structure satisfies, consequently
\[
    \Pi^\sharp = i \circ (-\omega^\sharp) \circ i^*,
\]
i.e., the diagram
\begin{equation} \label{eq:anchor-maps-reg-poisson}
    \begin{tikzcd}
    {\mathrm{T}^* M} \arrow[r, "\Pi^\sharp"] \arrow[d, "i^*"] & {\mathrm{T}M} \\
    {\mathrm{T}^* \mathcal{F}} \arrow[r, "-\omega^\sharp"] & {\mathrm{T} \mathcal{F}} \arrow[u, "i"]
    \end{tikzcd}
\end{equation}
commutes.

\subsection{Equivariant cohomology}
\label{ssec:equiv-cohomology}

Equivariant cohomology is a refinement of ordinary singular cohomology, first introduced by Borel \cite{borel_1960_seminar} to compute the cohomology of a general quotient space $M/G$ even when the group action is poorly behaved. The essential idea is to replace $M$ by the product $M \times \mathrm{E}G$, which is weakly homotopically equivalent to $M$, and then consider the \emph{homotopy quotient} $M_G \coloneqq M \times_G \mathrm{E}G$. Here, $\mathrm{E}G$ is a weakly contractible space which carries a free $G$-action. The \emph{equivariant cohomology} of $M$ is taken to be the cohomology of the homotopy quotient, i.e.,
\begin{equation} \label{eq:def-equiv-cohomology-homotopy}
    \mathrm{H}_G^\bullet(M) \coloneqq \mathrm{H}^\bullet(M_G).
\end{equation}

Similarly to de~Rham's theorem in the non-equivariant case, there exists a differential complex $\Omega_G^\bullet(M)$ which computes the torsion-free part of equivariant cohomology (cf. \cite[Thm.\,~A.1]{tu_introductory_2020}). In the assumptions we need that the Lie group is compact and connected. This description is known as \emph{Cartan's model of equivariant cohomology}, and the cochain complex is explicitly given by
\begin{equation} \label{eq:cartan-complex-def}
    \Omega_G^\bullet(M) \coloneqq \big( \Omega^\bullet(M) \otimes \operatorname{Sym}^\bullet(\mathfrak{g}^*) \big)^G.
\end{equation}
We define the grading of the previous complex by assigning $\deg \omega = i + 2j$ for $\omega \in \Omega^i(M) \otimes \operatorname{Sym}^j(\mathfrak{g}^*)$. The $G$-action on $\Omega^\bullet(M)$ is given by the pullback of differential forms, the $G$-action on $\operatorname{Sym}^\bullet(\mathfrak{g}^*)$ is given by the coadjoint action, and the action on the tensor product is induced from the previous ones. Elements of the complex $\Omega_G^\bullet(M)$ are known as \emph{equivariant differential forms} and we will usually distinguish them from regular de Rham forms by means of an underscore, i.e., we write an equivariant form as $\underline{\alpha} \in \Omega^\bullet_G(M)$.

To describe the equivariant differential, let us fix a basis $X_1, \ldots, X_m$ of $\mathfrak{g}$ and let $\sigma^1, \ldots, \sigma^m$ be a dual basis. If $X \in \mathfrak{g}$, we denote by $X^\#\in \mathfrak{X}(M)$ its fundamental vector field. In this notation, the equivariant differential acts on an element $\omega \otimes p \in \Omega_G^\bullet(M)$ as
\begin{equation} \label{eq:equiv-diff-elem-def}
    \diff_G(\omega \otimes p) = \diff \omega \otimes p - \sum_{i = 1}^m (\iota_{X_i^\#} \omega) \otimes (\sigma^i p)
\end{equation}
and is extended by linearity. In a different notation,
\begin{equation} \label{eq:equiv-diff-tensor-def}
    \diff_G = \diff \otimes \operatorname{id} - \sum_{i = 1}^m \iota_{X_i^\#} \otimes \sigma^i.
\end{equation}

We can abstract the previous construction to obtain a broader definition of equivariant cohomology. We follow the presentation of \cite[Sec.\,2.1]{ginzburg_equivariant_1999}. Let $(A^\bullet,\diff)$ be a cochain complex of Fréchet spaces equipped with a smooth $G$-action commuting with $\mathrm{d}$. Denote by $\mathcal{L}_X\colon A^\bullet \to A^\bullet$ the infinitesimal $G$-action generated by $X\in\mathfrak{g}$. We say that $(A^\bullet,\mathrm{d})$ is a \emph{$G$-differential complex} if there exists a linear map $\mathfrak{g} \to \operatorname{End}^{-1}(A^\bullet)$, $X \mapsto \iota_X$, such that
\begin{subequations}
\begin{align}
    \iota_X\iota_Y+\iota_Y\iota_X &= 0, \\
    g\iota_Xg^{-1} &= \iota_{\operatorname{Ad}_g X}, \\
    \mathcal{L}_X &= \diff \iota_X + \iota_X \diff.
\end{align}
\end{subequations}

Given a $G$-differential complex, its Cartan complex is
\[
    \Omega_G^\bullet(A) \coloneqq \left(A^\bullet \otimes \operatorname{Sym}^\bullet\mathfrak{g}^*\right)^G,
\]
where elements of $\operatorname{Sym}^k\mathfrak{g}^*$ are assigned degree $2k$ and its differential is defined by expression \eqref{eq:equiv-diff-elem-def}. On the $G$-invariant subcomplex, one has $\diff_G^2=0$. This complex and its differential are called the \emph{Cartan complex}, or \emph{equivariant complex}, and the \emph{equivariant differential} of $A^\bullet$, respectively. Its cohomology groups are called the \emph{equivariant cohomology groups} of $A^\bullet$ and are denoted by $\mathrm{H}_G^\bullet(A)$.

Let us now consider a Poisson manifold $(M,\Pi)$ equipped with a Poisson action $\rho$ of a compact connected Lie group $G$, and denote by $\tau\colon\mathfrak{g}\to\mathfrak{X}^1(M)$ the corresponding infinitesimal action. The infinitesimal action alone does not provide the degree $-1$ contraction operators required to make the Poisson complex $(\mathfrak{X}^\bullet(M),\diff_\Pi)$ into a $G$-differential complex. To construct such operators, one introduces an equivariant pre-momentum map $\lambda \colon \mathfrak{g} \to \Omega^1(M)$ satisfying $\tau(X)=\Pi^\sharp (\lambda(X) )$ for all $X\in\mathfrak{g}$. Thus, the infinitesimal action $\tau$ and the pre-momentum map $\lambda$ are related by the anchor map $\Pi^\sharp\colon \mathrm{T}^*M \to \mathrm{T}M$. The natural pairing between differential forms and multivector fields then defines the contraction operators associated with the resulting $G$-differential structure.

\begin{proposition}[{Ginzburg~\cite[Sec.\,4.1]{ginzburg_equivariant_1999}}] \label{prop:equiv-poisson-dgga}
    In the previous assumptions, the maps $\diff_\Pi$, $\iota_X \coloneqq \iota_{\lambda(X)} \colon \mathfrak{X}^{\bullet}(M) \to \mathfrak{X}^{\bullet - 1}(M)$ and $\mathcal{L}_X \coloneqq \mathcal{L}_{\tau(X)} \colon \mathfrak{X}^\bullet(M) \to \mathfrak{X}^\bullet(M)$ endow the Lichnerowicz complex $(\mathfrak{X}^\bullet(M), \diff_\Pi)$ with the structure of a differential graded $\mathfrak{g}$-algebra.
\end{proposition}

From the previous proposition, the equivariant Poisson cohomology groups $\mathrm{H}_{\Pi, G}^\bullet(M)$ can be defined following the general description of equivariant cohomology for Fréchet $G$-differential complexes, and the equivariant Poisson differential is explicitly given by
\begin{equation}
    \diff_{\Pi, G} = \diff_\Pi \otimes \operatorname{id}_{\operatorname{Sym}^\bullet(\mathfrak{g}^*)} - \sum_{i = 1}^m \iota_{\lambda(X_i)} \otimes \sigma^i.
\end{equation}

\section{Equivariant co-orientable, codimension one foliations} \label{sec:equiv-coorient-codim1-foliations}

In order to study the equivariant counterparts of the obstruction classes of Guillemin, Miranda, and Pires, we must first explore the properties of the objects they characterize. In this section we begin by discussing the notion of equivariant, co-orientable, codimension-one foliations and their cohomological properties.

\subsection{Preliminaries on co-orientable, codimension-one foliations}

Let us consider a codimension-one foliation $\mathcal{F}$ on $M$: the inclusion of bundles $i \colon \mathrm{T} \mathcal{F} \to \mathrm{T} M$ gives rise to the short exact sequence
\begin{equation} \label{eq:ses-tangent-foliation}
    \begin{tikzcd}[column sep=small]
    	0 & {\mathrm{T} \mathcal{F}} & {\mathrm{T} M} & {\mathrm{N} \mathcal{F}} & 0
    	\arrow[from=1-1, to=1-2]
    	\arrow["i", from=1-2, to=1-3]
    	\arrow[from=1-3, to=1-4]
    	\arrow[from=1-4, to=1-5]
    \end{tikzcd}
\end{equation}
We say $\mathcal{F}$ is \emph{co-orientable} if $\mathrm{N} \mathcal{F}$ is a trivial line bundle. By looking at the dual short exact sequence
\begin{equation} \label{eq:ses-cotangent-foliation}
    \begin{tikzcd}[column sep=small]
    	0 & {\mathrm{N}^* \mathcal{F}} & {\mathrm{T}^* M} & {\mathrm{T}^* \mathcal{F}} & 0
    	\arrow[from=1-1, to=1-2]
    	\arrow[from=1-2, to=1-3]
    	\arrow["j", from=1-3, to=1-4]
    	\arrow[from=1-4, to=1-5]
    \end{tikzcd}
\end{equation}
we observe $\mathcal{F}$ is co-orientable if and only if $\mathrm{N}^* \mathcal{F}$ is trivial or, in other words, there exists a nowhere vanishing form $\alpha \in \Omega^1(M)$ such that $\mathcal{F} = \ker \alpha$. Notice that Frobenius' theorem implies we must have
\begin{equation} \label{eq:normal-var-alpha}
    \diff \alpha = \alpha \wedge \nu.
\end{equation}

The short exact sequence of bundles \eqref{eq:ses-cotangent-foliation} induces a short exact sequences of cochain complexes
\begin{equation} \label{eq:ses-cochain-foliated}
    \begin{tikzcd}[column sep=small]
    	0 & {\Omega^\bullet(\mathcal{F}; \mathrm{N}^* \mathcal{F})} & {\Omega^\bullet(M)} & {\Omega^\bullet(\mathcal{F})} & 0
    	\arrow[from=1-1, to=1-2]
    	\arrow[from=1-2, to=1-3]
    	\arrow["j", from=1-3, to=1-4]
    	\arrow[from=1-4, to=1-5]
    \end{tikzcd}
\end{equation}
where $\Omega^\bullet(\mathcal{F}; \mathrm{N}^* \mathcal{F})$ is the complex of foliated forms with coefficients in the conormal bundle. This complex is endowed with the differential $\diff_\nabla$, which is induced from the map $\mathrm{T} \mathcal{F} \to \mathrm{N}^* \mathcal{F}$ given by the Bott connection. This short exact sequence induces a long exact sequence in the corresponding cohomology groups.

To relate this description to the invariants of Guillemin, Miranda, and Pires we choose an explicit trivialization of $\mathrm{N}^* \mathcal{F}$, which amounts to the choice of a defining one-form $\alpha \in \Omega^1(M)$ for $\mathcal{F}$. A choice of splitting for \eqref{eq:ses-cotangent-foliation} is equivalent to a choice of splitting for \eqref{eq:ses-tangent-foliation}, which is tantamount to a retraction $r \colon \mathrm{N} \mathcal{F} \to \mathrm{T}M$ or, equivalently, to a transverse vector field $R \in \mathfrak{X}(M)$ to $\mathcal{F}$. Let us choose any such vector field satisfying $\alpha(R) = 1$.

The induced section at the level of cotangent bundles is given by
\begin{equation} \label{eq:section}
    \begin{array}{rccc}
        s \colon & \mathrm{T}^* \mathcal{F} & \longrightarrow & \mathrm{T}^* M \\
         & [\eta] & \longmapsto & \eta - \alpha \, \iota_R \eta
    \end{array}.
\end{equation}
The map does not depend on the choice of representative since taking $\eta, \eta' \in \mathrm{T}^* M$ such that $[\eta] = [\eta']$, which is equivalent to $\eta' - \eta = f \alpha$ for some $f \in \mathcal{C}^\infty(M)$, we have that
\begin{equation*}
    \eta' - \alpha \iota_R \eta' = \eta + \alpha f - \alpha \iota_R( \eta + \alpha f ) = \eta + \alpha f - \alpha \iota_R \alpha - \alpha f = \eta - \alpha \iota_R \eta.
\end{equation*}
Observe that such a splitting induces an analogous splitting of cochain complexes
\begin{equation} \label{eq:splitting-cochain-foliated}
    \Omega^\bullet(M) \cong \alpha \wedge \Omega^{\bullet - 1}(\mathcal{F}) \oplus \Omega^\bullet(\mathcal{F}).
\end{equation}

We are now in position to describe the structure cochain complex and the induced long exact sequence in cohomology under this choice of trivialization.

\begin{proposition}
    In the previous assumptions, we have
    \begin{equation}
        (\diff s - s \diff_{\mathcal{F}})(\eta) = \alpha \wedge ( \nu \wedge \eta - \mathcal{L}_R \eta).
    \end{equation}
\end{proposition}

\begin{proof}
    Direct computation with expression \eqref{eq:section} shows that
    \begin{align*}
        s \diff_{\mathcal{F}} \eta - \diff s \eta &= s( \diff \eta ) - \diff (\eta - \alpha \iota_R \eta) \\
        &= \diff \eta - \alpha \wedge \iota_R \diff \eta - \diff \eta + \diff \alpha \, \iota_R \eta - \alpha \wedge \diff \iota_R \eta \\
        &= - \alpha \, \iota_R \diff \eta + \alpha \wedge \nu \, \iota_R \eta - \alpha \wedge \iota_R \eta \\
        &= \alpha \wedge (\nu \iota_R \eta - \mathcal{L}_R \eta). \qedhere
    \end{align*}
\end{proof}

\begin{corollary} \label{cor:connecting-morphism}
    In the previous assumptions, we have that:
    \begin{enumerate}
        \item under the trivialization \eqref{eq:section}, the twisted differential is given by $\diff_\nabla = - \diff_{\mathcal{F}} + j\nu \wedge$, and
        \item the connecting morphism of the long exact sequence is given by $\delta = \mathcal{L}_R - j\nu \wedge \iota_R$.
    \end{enumerate}
\end{corollary}

\begin{proof}
    For the first item, let us express and arbitrary element in $\alpha \wedge \Omega^\bullet(\mathcal{F})$ as $\alpha \wedge s(\omega)$. Since the sequence \eqref{eq:ses-cochain-foliated} is exact, the operator $\diff_\nabla$ is simply the restriction of $\diff$ to the given subcomplex. We consequently have
    \begin{align*}
        \diff(\alpha \wedge s(\omega)) &= \diff \alpha \wedge s(\omega) - \alpha \wedge \diff s(\omega) \\
        &= \alpha \wedge \nu \wedge s(\omega) - \alpha \wedge \big(s \diff_{\mathcal{F}} \omega  - \alpha \wedge (\nu\, \iota_R \omega - \mathcal{L}_R \omega) \big) \\
        &= \alpha \wedge s(j\nu \wedge \omega - \diff_{\mathcal{F}} \omega).
    \end{align*}

    For the second item, recall that the connecting morphism is defined as $\delta(\omega) = i(\diff \omega')$ for any $\omega'$ satisfying $j(\omega') = \omega$. In our case, we may simply consider $\omega' = s(\omega)$ and obtain therefore
    \begin{equation*}
        \diff s(\omega) = s \diff_{\mathcal{F}} \omega - \alpha \wedge (\nu \wedge \iota_R \omega - \mathcal{L}_R \omega) = \alpha \wedge (\mathcal{L}_R \omega - j\nu \wedge \iota_R \omega). \qedhere
    \end{equation*}
\end{proof}

Consequently, under the choices made so far the long exact sequence in cohomology associated to the sequence \eqref{eq:ses-cochain-foliated} specializes to
\begin{equation}
    \begin{tikzcd}[column sep=small]
    	\cdots & {\mathrm{H}^{\bullet - 1}(\mathcal{F})} & {\mathrm{H}^{\bullet - 1}(\mathcal{F}, \diff_\nabla)} & {\mathrm{H}^\bullet(M)} & {\mathrm{H}^\bullet(\mathcal{F})} & {\mathrm{H}^\bullet(\mathcal{F}, \diff_\nabla)} & \cdots
    	\arrow[from=1-1, to=1-2]
    	\arrow[from=1-2, to=1-3]
    	\arrow[from=1-3, to=1-4]
    	\arrow["j", from=1-4, to=1-5]
    	\arrow[from=1-5, to=1-6]
    	\arrow[from=1-6, to=1-7]
    \end{tikzcd}
\end{equation}

\begin{remark}
    Observe that, from equation \eqref{eq:normal-var-alpha}, the form $\nu$ is the normal variation of $\alpha$ and therefore modified differential corresponds precisely to the induced differential from the Bott connection.
\end{remark}

\begin{remark}
    The introduction of $\alpha$ as a choice of trivialization of $\mathrm{N}^* \mathcal{F}$ has been used to give a precise description of the cohomology groups appearing in \eqref{eq:ses-cochain-foliated}. However, we should prove that the resulting cohomology groups are independent of this choice.

    If $\alpha$ and $\alpha'$ are two different choices of trivialization, we know $\alpha' = f \alpha$ for some nowhere vanishing, $G$-invariant function $f$. This implies
    \begin{equation}
        \alpha' \wedge \nu' = \diff \alpha' = \diff(f \alpha) = \diff f \wedge \alpha + f \diff \alpha = \frac{\diff f}{f} \wedge \alpha' + \alpha' \wedge \nu = \alpha' \wedge (\nu - \diff \log \lvert f \rvert).
    \end{equation}
    Consequently, $j\nu' = j\nu - \diff_{\mathcal{F}} \log \lvert f \rvert$ and therefore the induced operators in cohomology are isomorphic. The term $\diff \log \lvert f \rvert$ appears explicitly in the computations of Guillemin, Miranda, and Pires to show that the given obstruction classes are independent of the choice of defining function~\cite{guillemin_codimension_2011}.
\end{remark}

\subsection{Equivariant co-orientable, codimension-one foliations} \label{ssec:equiv-coorient-codimone-fol}

Given we are interested in equivariant cosymplectic structures and the equivariant counterpart of the obstruction classes of Guillemin, Miranda, and Pires, we need to define first the equivariant counterpart of the foliations discussed thus far. Let us assume from now and onward that we have a manifold $M$ acted upon by a compact Lie group $G$.

\begin{definition}
    An \emph{equivariant co-orientable, codimension-one foliation} is a distribution $\mathcal{D} \subseteq \mathrm{T}M$ defined by a nowhere vanishing one-form $\underline{\alpha} \in \Omega^1(M)$ satisfying $\diff_G \underline{\alpha} = \underline{\alpha} \wedge \underline{\nu}$.
\end{definition}

Some care is needed when identifying these foliations with their non-equivariant counterparts. Observe that, since $\Omega_G^1(M) = \Omega^1(M)^G$, we have $\underline{\alpha} = \alpha$ for some $G$-invariant de Rham form. For the same reasons, $\underline{\nu} = \nu$. Therefore, $\mathcal{D}$ is a co-orientable, codimension one distribution. The condition $\diff_G \underline{\alpha} = 0$ in this identification reads now
\begin{equation*}
    \alpha \wedge \nu = \diff_G \alpha = \diff \alpha - \sum_{i = 1}^m (\iota_{X_i^\#} \alpha) \otimes \sigma^i.
\end{equation*}
Since both terms have different degree, we arrive to the following conclusions:
\begin{itemize}
    \item $\diff \alpha = \alpha \wedge \nu$, thus $\mathcal{D}$ is an involutive distribution, and
    \item $\iota_{X^\#} \alpha = 0$ for very $X \in \mathfrak{g}$.
\end{itemize}
The latter condition implies that, whenever $G$ is connected, the $G$-action preserves all leaves of $\mathcal{F}$ individually. In general, $G$ only preserves the distribution $\mathcal{D}$.

Under the previous assumptions, the equivariant differential $\diff_G$ factorizes through the projection $j \otimes \operatorname{id}_{\operatorname{Sym}^\bullet(\mathfrak{g}^*)}$ and thus induces an equivariant foliated differential $\diff_{\mathcal{F}, G}$. The chain complex $(\Omega_G^\bullet(\mathcal{F}), \diff_{\mathcal{F}, G})$ is called the complex of \emph{equivariant foliated forms}. In the language of \Cref{ssec:equiv-cohomology}, it is induced from the $G$-differential complex $\Omega^\bullet(\mathcal{F})$ with the restriction of the operators $\iota_X$ and $\mathcal{L}_X$. Similarly, we have an analogous complex $\Omega_G^\bullet(\mathcal{F}; \mathrm{N}^* \mathcal{F})$ with differential $\diff_{\nabla, G}$.

Our main objective is to establish an analogue of the short exact sequence \eqref{eq:ses-cochain-foliated} ---and, similarly, to \eqref{eq:splitting-cochain-foliated}--- in this equivariant setting. Observe that the Cartan complex is obtained by the composition of the functors\footnote{
    In our notation we use the question mark ``?'' as a placeholder in the expression of certain functors.
} $? \otimes_{\mathbb{R}} \operatorname{Sym}^\bullet(\mathrm{g}^*)$ and $?^G$. While the former is exact, since we are taking tensor products over a field, the latter is only left exact in general. To overcome this problem we show that the splittings described in the previous section can be chosen to be $G$-invariant, proving the desired result.

\begin{proposition} \label{prop:equiv-ses-cochain}
    Let us consider an equivariant, co-orientable, codimension-one foliation in $M$. If $R$ is a $G$-invariant vector field $R \in \mathfrak{X}(M)$ such that $\alpha(R) = 1$, the section \eqref{eq:section} descends to a map of complexes
    \begin{equation*}
        s \otimes \operatorname{id}_{\operatorname{Sym}^\bullet(\mathfrak{g}^*)} \colon \Omega_G^\bullet( \mathcal{F} ) \longrightarrow \Omega_G^\bullet(M),
    \end{equation*}
    which induces a section of the map
    \begin{equation*}
        j \otimes \operatorname{id}_{\operatorname{Sym}^\bullet(\mathfrak{g}^*)} \colon \Omega_G^\bullet( M ) \longrightarrow \Omega_G^\bullet( \mathcal{F}).
    \end{equation*}

    Consequently,the following short sequences are exact:
    \begin{equation} \label{eq:ses-cochain-equiv-foliated}
        \begin{tikzcd}
        	0 & {\Omega^\bullet_G(\mathcal{F}; \mathrm{N}^* \mathcal{F})} & {\Omega_G^\bullet(M)} & {\Omega_G^\bullet(\mathcal{F})} & 0
        	\arrow[from=1-1, to=1-2]
        	\arrow[from=1-2, to=1-3]
        	\arrow[from=1-3, to=1-4]
        	\arrow[from=1-4, to=1-5]
        \end{tikzcd}
    \end{equation}
    \begin{equation} \label{eq:ses-cochain-equiv-foliated-trivialized}
        \begin{tikzcd}
        	0 & {\alpha \wedge \Omega^{\bullet - 1}_G(\mathcal{F})} & {\Omega_G^\bullet(M)} & {\Omega_G^\bullet(\mathcal{F})} & 0
        	\arrow[from=1-1, to=1-2]
        	\arrow[from=1-2, to=1-3]
        	\arrow[from=1-3, to=1-4]
        	\arrow[from=1-4, to=1-5]
        \end{tikzcd}
    \end{equation}
\end{proposition}

\begin{proof}
    Observe we may always choose a $G$-invariant vector field $R \in \mathfrak{X}(M)$ satisfying $\alpha(R) = 1$: since $\alpha$ is $G$-invariant and $G$ is compact, we may average the expression $\alpha(R) = 1$ for any vector field $R$. Consequently, the section $s$ in equation \eqref{eq:section} can be chosen to be $G$-invariant, and therefore it descends to a map of complexes
    \begin{equation*}
        s \otimes \operatorname{id}_{\operatorname{Sym}^\bullet(\mathfrak{g}^*)} \colon \Omega_G^\bullet(\mathcal{F}) \longrightarrow \Omega_G^\bullet(M).
    \end{equation*}
    Furthermore, since $js = \operatorname{id}_{\Omega^\bullet(\mathcal{F})}$ we have
    \begin{equation*}
        (j \otimes \operatorname{id}_{\operatorname{Sym}^\bullet(\mathfrak{g}^*)}) \circ (s \otimes \operatorname{id}_{\operatorname{Sym}^\bullet(\mathfrak{g}^*)}) = (js) \otimes \operatorname{id}_{\operatorname{Sym}^\bullet(\mathfrak{g}^*)} = \operatorname{id}_{\Omega^\bullet(\mathcal{F})} \otimes \operatorname{id}_{\operatorname{Sym}^\bullet(\mathfrak{g}^*)} = \operatorname{id}_{\Omega_G^\bullet(\mathcal{F})}
    \end{equation*}
    and hence the short sequence \eqref{eq:ses-cochain-equiv-foliated} is exact since we have found an explicit splitting.

    The proof that all the previous identifications in \Cref{cor:connecting-morphism} apply to this case are analogous, since the only difference is that all maps have to be tensored with the identity operator $\operatorname{id}_{\operatorname{Sym}^\bullet(\mathfrak{g}^*)}$.
\end{proof}

\section{Equivariant obstruction classes for codimension-one symplectic foliations} \label{sec:equivariant-obstruction-classes}

The classical obstruction classes in \cite{guillemin_codimension_2011} provide necessary and sufficient conditions for a co-orientable, codimension-one, symplectic foliation to arise from a cosymplectic structure. In this section we study the equivariant version of this problem. We begin by making precise the notion of equivariant cosymplectic structure.

\begin{definition} \label{def:equivariant-cosymplectic}
    Let $M$ be a smooth manifold with $\dim M = 2n + 1$ endowed with an action of a compact, connected Lie group $G$. An \emph{equivariant cosymplectic structure} is a pair $(\underline{\alpha}, \underline{\beta})$, with $\underline{\alpha} \coloneqq \alpha \in \Omega_G^1(M)$ and $\underline{\beta} \coloneqq \beta - \mu \in \Omega_G^2(M)$ equivariantly closed forms satisfying the condition $\alpha \wedge \beta^n \neq 0$.
\end{definition}

Observe that, since $\diff_G \alpha = 0$, equivariant cosymplectic structures are an instance of equivariant, co-orientable, codimension-one foliations and therefore the results in \Cref{ssec:equiv-coorient-codimone-fol} apply. Moreover, since $\nu = 0$ we have that $\diff_{\nabla, G} = - \diff_{\mathcal{F}, G}$ and then the short exact sequence \eqref{eq:ses-cochain-equiv-foliated} induces the following long exact sequence of equivariant cohomology groups:
\begin{equation}
    \begin{tikzcd}[column sep=small]
    	\cdots & {\mathrm{H}_G^{\bullet - 1}(\mathcal{F})} & {\mathrm{H}_G^{\bullet - 1}(\mathcal{F})} & {\mathrm{H}_G^\bullet(M)} & {\mathrm{H}^\bullet_G(\mathcal{F})} & {\mathrm{H}^\bullet_G(\mathcal{F})} & \cdots
    	\arrow[from=1-1, to=1-2]
    	\arrow[from=1-2, to=1-3]
    	\arrow[from=1-3, to=1-4]
    	\arrow["j", from=1-4, to=1-5]
    	\arrow[from=1-5, to=1-6]
    	\arrow[from=1-6, to=1-7]
    \end{tikzcd}
\end{equation}

More can be said about the structure of equivariant cosymplectic manifolds. Since $\diff_G \underline{\beta} = 0$ we have
\begin{equation} \label{eq:equiv-fol-symplectic-closed}
    0 = \diff_G (\beta - \mu_i \otimes \sigma^i) = \diff \beta - \sum_{i = 1}^m (\diff \mu_i + \iota_{X_i^\#} \beta) \otimes \sigma^i.
\end{equation}
Checking degrees individually we obtain $\diff \beta = 0$, and consequently we have proved

\begin{proposition}
    Let $M$ be a smooth manifold with an equivariant cosymplectic structure $(\alpha, \beta - \mu)$. Then, the triple $(M, \alpha, \beta)$ is a cosymplectic manifold or, in other words, we have $\diff \alpha = 0$, $\diff \beta = 0$, and $\alpha \wedge \beta^n \neq 0$. Moreover, the forms $\alpha$ and $\beta$ are $G$-invariant.
\end{proposition}

\begin{remark}
    The second term in equation \eqref{eq:equiv-fol-symplectic-closed} suggests that the elements $\mu_i$ appearing in the definition of $\underline{\beta}$ can be interpreted as moment maps for the fundamental action of $G$. This is indeed the case, and the relation between equivariant cosymplectic structures, Hamiltonian actions, and cosymplectic reduction are studied in \cite{miranda_hamiltonian_2026}.
\end{remark}

Recall that the identifications in \Cref{sec:equiv-coorient-codim1-foliations} depend on the choice of a $G$-invariant vector field $R \in \mathfrak{X}(M)$ satisfying $\alpha (R) = 1$. In the setting of cosymplectic manifolds there exists a particularly natural choice for $R$: we call the \emph{Reeb field} the unique vector field $R \in \mathfrak{X}(M)$ satisfying the system of equations
\begin{equation}
    \iota_R \alpha = 1, \qquad \iota_R \beta = 0.
\end{equation}
Notice that, for an equivariant cosymplectic structure, the fact that both $\alpha$ and $\beta$ are $G$-invariant together with the uniqueness of $R$ implies that $R$ is $G$-invariant.

The Reeb field has one additional important property: for the Poisson tensor $\Pi$ associated to the symplectic foliation $(\ker \alpha, \beta)$ following \Cref{ssec:poisson-geometry}, the Reeb field $R$ is a Poisson vector field, i.e., $\mathcal{L}_R \Pi = 0$.

\subsection{The equivariant obstruction classes}

After the description of the equivariant cohomology groups in \Cref{prop:equiv-ses-cochain} we are in position to obtain the cohomological invariants of Guillemin, Miranda, and Pires in the equivariant setting.

\begin{theorem} \label{thm:first-obstruction-class}
    Let us consider an equivariant, co-orientable, codimension-one foliation $\mathcal{F}$. Let us fix a defining form $\alpha \in \Omega^1_G(M)$ and denote the corresponding normal variation in \eqref{eq:normal-var-alpha} by $[j \nu] \in \mathrm{H}_G^1(\mathcal{F})$.

    Then, $\mathcal{F}$ admits an equivariantly closed defining one-form if and only if $c^G_{\mathcal{F}} \coloneqq [j \nu] \in \mathrm{H}_G^1(\mathcal{F})$ vanishes.
\end{theorem}

\begin{proof}
    Let us begin by observing that, from the exactness of the short exact sequence \eqref{eq:ses-cochain-equiv-foliated}, we have the commutative diagram
    \begin{equation*}
        \begin{tikzcd}[column sep=scriptsize, row sep=2.25em]
        	0 & {\Omega_G^2(\mathcal{F}; \mathrm{N}^* \mathcal{F})} & {\Omega_G^2(M)} \\
        	0 & {\Omega_G^1(\mathcal{F}, \mathrm{N}^* \mathcal{F})} & {\Omega_G^1(M)}
        	\arrow[from=1-1, to=1-2]
        	\arrow[from=1-2, to=1-3]
        	\arrow[from=2-1, to=2-2]
        	\arrow["{\diff_{\nabla, G}}", from=2-2, to=1-2]
        	\arrow[from=2-2, to=2-3]
        	\arrow["{\diff_G}"', from=2-3, to=1-3]
        \end{tikzcd}
    \end{equation*}
    Observe, moreover, that since $\Omega^1_G(\mathcal{F}; \mathrm{N}^* \mathcal{F}) \cong \alpha \, \mathcal{C}^\infty(M)^G$, all possible equivariant defining one-forms for the foliation $\mathcal{F}$ naturally reside in $\Omega^1_G(\mathcal{F}; \mathrm{N}^* \mathcal{F})$. Since the inclusion morphisms $\Omega^\bullet_G(\mathcal{F}; \mathrm{N}^* \mathcal{F}) \to \Omega_G^\bullet(M)$ are injective, we have $\diff_G(f \alpha) = 0$ if and only if $\diff_{\nabla, G} f = 0$.

    We are now in position to prove the result. From the previous discussion, there exists an equivariantly closed defining one-form if and only if we may find a nowhere-vanishing function $f \in \mathcal{C}^\infty(M)^G$ such that $0 = \diff_{\nabla, G} f = - \diff_{\mathcal{F}, G} f + f j\nu$. However, this equation is equivalent to $j \nu = \diff_{\mathcal{F}, G} \log \lvert f \rvert$, which concludes the claim.

    To prove the theorem we only need to show that the form $j \nu \in \Omega^1_G(\mathcal{F})$ is $\diff_{\mathcal{F},G}$-closed. However, since $\diff_G \alpha = \alpha \wedge j\nu$ by definition and $\diff_G^2 = 0$, we have
    \begin{equation*}
        0 = \diff_\nabla (j \nu) = - \diff_{\mathcal{F}, G} (j \nu) + j\nu \wedge j\nu = - \diff_{\mathcal{F}, G} (j \nu)
    \end{equation*}
    and hence $j\nu$ is $\diff_{\mathcal{F}, G}$-closed
\end{proof}

\begin{theorem} \label{thm:second-obstruction-class}
    Let $(M, \mathcal{F}, \beta)$ be an equivariant, co-orientable, codimension-one symplectic foliation. Then, the symplectic foliated form $\beta \in \Omega^2_G(M)$ admits an equivariantly closed representative $\beta' \in \Omega_G^2(M)$ if and only if the cohomology class $\sigma^G_{\mathcal{F}} \coloneqq \delta[\beta] \in \mathrm{H}_G^3(\mathcal{F}; \mathrm{N}^* \mathcal{F})$ vanishes.
\end{theorem}

\begin{proof}
    Since $\beta \in \Omega^2_G(\mathcal{F})$ is $\diff_{\mathcal{F},G}$-closed we obtain an element in the cohomology group $[\beta] \in \mathrm{H}_G^2(\mathcal{F})$. By the long exact sequence in cohomology
    \begin{equation*}
        \begin{tikzcd}[column sep=small]
        	\cdots & {\mathrm{H}^2_G(M)} & {\mathrm{H}^2_G(\mathcal{F})} & {\mathrm{H}^3_G(\mathcal{F}; \mathrm{N}^* \mathcal{F})} & \cdots
        	\arrow[from=1-1, to=1-2]
        	\arrow[from=1-2, to=1-3]
        	\arrow["\delta", from=1-3, to=1-4]
        	\arrow[from=1-4, to=1-5]
        \end{tikzcd}
    \end{equation*}
    we observe $[\beta] = j^*[\beta']$ for some $\beta' \in \mathrm{H}_G^2(M)$ if and only if $[\beta] \in \ker \delta$. Therefore, the element $\delta[\beta] \in \mathrm{H}_G^3(\mathcal{F}; \mathrm{N}^* \mathcal{F})$ is the obstruction class we are looking for.
\end{proof}

In the original computations of Guillemin, Miranda, and Pires, it is assumed that the first obstruction class given by \Cref{thm:first-obstruction-class} vanishes, which is given by $[j\nu] = 0 \in \mathrm{H}^1_G(\mathcal{F})$. Under these assumptions, we have an identification $\mathrm{H}^\bullet_G(\mathcal{F}; \mathrm{N}^* \mathcal{F}) \cong \mathrm{H}_G^{\bullet - 1}(\mathcal{F})$ and our obstruction class is defined in the cohomology group $\mathrm{H}_G^2(\mathcal{F})$. Consequently, our computations reduce to those of Guillemin, Miranda, and Pires, and extend them if the first obstruction class does not vanish. In particular, we obtain the following:

\begin{corollary} \label{cor:second-obstruction-class}
    Let $(M, \mathcal{F}, \beta)$ be an equivariant, co-orientable, codimension-one symplectic foliation and assume the first invariant $c_{\mathcal{F}}^G \in \mathrm{H}^1(\mathcal{F})$ vanishes (cf. \Cref{thm:first-obstruction-class}). Then, the symplectic foliated form $\beta \in \Omega^2_G(M)$ admits an equivariantly closed representative $\beta' \in \Omega_G^2(M)$ if and only if the cohomology class $\sigma^G_{\mathcal{F}} \coloneqq \delta[\beta] \in \mathrm{H}_G^2(\mathcal{F})$ vanishes.
\end{corollary}

\subsection{Equivariant unimodularity for cosymplectic manifolds}

Let us now consider a co-orientable, codimension-one symplectic foliation $(M, \mathcal{F}, \beta)$ and let $\Pi_{\mathcal{F}}$ and $\Pi$ be the Poisson structures defined following \Cref{ssec:poisson-geometry}. In \cite[Thm.\,10]{guillemin_codimension_2011}, Guillemin, Miranda, and Pires provide the following criterion for unimodularity: the Poisson manifold $(M, \Pi)$ is unimodular if and only if the first obstruction class $c_{\mathcal{F}}^G \in \mathrm{H}_{G}^1(\mathcal{F})$ vanishes.

We will prove an analogous result in the setting of equivariant unimodularity. In order to build the equivariant Poisson complex, we have to find a cotangent lift of the group action following \Cref{prop:equiv-poisson-dgga}.

\begin{proposition}
    Let $(M, \mathcal{F}, \beta)$ and let $\Pi_{\mathcal{F}}$ and $\Pi$ be the Poisson structures defined following \Cref{ssec:poisson-geometry}. Let $\rho$ be a $G$-action tangent to $\mathcal{F}$ for a compact, connected Lie group $G$ and let $\tau \colon \mathfrak{g} \to \mathfrak{X}(\mathcal{F})$ be the infinitesimal action.
    
    Then, the maps $\lambda = - \beta^\flat \tau \colon \mathfrak{g} \to \Omega^1(\mathcal{F})$ and $s \lambda \colon \mathfrak{g} \to \Omega^1(M)$, with $s$ defined in \eqref{eq:section}, are cotangent lifts, respectively, for $\tau$ and $i \tau$. In other words, the following equations hold:
    \begin{subequations}
        \begin{align}
            \tau &= \Pi_{\mathcal{F}}^\sharp \lambda, \label{eq:cot-lift-fol} \\
            i \tau &= \Pi^\sharp s \lambda. \label{eq:cot-lift}
        \end{align}
    \end{subequations}
\end{proposition}

\begin{proof}
    For the proof of \eqref{eq:cot-lift-fol} recall that, from equation \eqref{eq:anchor-maps-reg-poisson} and the definition of cotangent lift, $\lambda$ must satisfy the relation $\tau = - \beta^\sharp \lambda$. Now, by non-degeneracy of $\beta$, this equation holds if and only if $\lambda = - \beta^\flat \tau$. This proves the first part.

    For the second part, observe that the section $s$ satisfies the relation $i^* s = \operatorname{id}$ by construction. Once again, direct computation with equation \eqref{eq:anchor-maps-reg-poisson} implies 
    \begin{equation*}
        \Pi^\sharp s \lambda = i \Pi^\sharp_{\mathcal{F}} i^* s \lambda = i \Pi^\sharp_{\mathcal{F}} \lambda = i \tau. \qedhere
    \end{equation*}
\end{proof}

We establish an analogous criterion for equivariant unimodularity. We assume henceforth that $\rho$ is a $G$-action which is tangent $\mathcal{F}$ and denote by $\tau \colon \mathfrak{g} \to \mathfrak{X}(\mathcal{F})$ the fundamental action. We begin by showing that the modular vector field $X^\Theta$ defines a representative in the equivariant Poisson cohomology group $\mathrm{H}_{\Pi, G}^1(M)$. 

\begin{proposition}
    In the previous assumptions, if $\Theta$ is $G$-invariant then the modular vector field $X^\Theta$ is also $G$-invariant. Moreover, it is $\diff_{\Pi, G}$-closed, and thus determines a cohomology class $[X^\Theta] \in \mathrm{H}_{\Pi, G}^1(M)$.
\end{proposition}

\begin{proof}
    First observe that, since $\Theta$ and $\Pi$ are $G$-invariant, the corresponding modular vector field $X^\Theta$ is also $G$-invariant. Therefore, it defines an element of the Cartan complex $\mathfrak{X}_G^1(M) = \mathfrak{X}^1(M)^G$.

    Let us now check that the modular vector field is $\diff_{\Pi, G}$-closed. Direct computation shows
    \begin{equation*}
        \diff_{\Pi, G} X^\Theta = \diff_\Pi X^\Theta - \sum_{i = 1}^m s \lambda(X_i) (X^\Theta) \otimes \sigma^i = - \sum_{i = 1}^m s \lambda(X_i) (X^\Theta) \otimes \sigma^i,
    \end{equation*}
    where we have used the fact that $\diff_\Pi X^\Theta = 0$ from the non-equivariant case. Let us prove now that $s \lambda(X) (X^\Theta)$ vanishes for any $X \in \mathfrak{g}$. Recall from Weinstein~\cite{weinstein_modular_1997} that, for a regular Poisson structure, the modular vector field is always tangent to the leaves of the symplectic foliation. In our case, since $\mathcal{F}_\Pi = \ker \alpha$, we have $s^*(X^\Theta) = X^\Theta - \alpha(X^\Theta) R = X^\Theta$ and therefore
    \begin{equation*}
        s \lambda(X) (X^\Theta) = \lambda(X)(X^\Theta) = \beta^\flat(\tau(X))(X^\Theta) = \beta(\tau(X), X^\Theta).
    \end{equation*}
    But remember now that, from the definition of Hamiltonian group action and that of the modular vector field, we have
    \begin{align*}
        \beta(\tau(X), X^\Theta) = (\iota_{X^\#} \beta) (X^\Theta) = - \langle \diff_{\mathcal{F}} \mu^X, X^\Theta \rangle &= - \langle \diff \mu^X, X^\Theta \rangle \\
        &= - \mathcal{L}_{X^\Theta} \mu^X = - \mathcal{L}_{X_{\mu^X}} \Theta / \Theta = - \mathcal{L}_{X^\#} \Theta / \Theta.
    \end{align*}
    Finally, because $\Theta$ is a $G$-invariant volume form, $\mathcal{L}_{X^\#} \Theta = 0$ and the claim is proved.
\end{proof}

The previous result shows that the modular vector field does define an equivariant Poisson cohomology class. With this preliminary in mind, we adapt the proof of Guillemin, Miranda, and Pires to the equivariant setting.

\begin{theorem}[Equivariant unimodularity]\label{thm:eq-unimod}
    In the previous assumptions, the Poisson manifold $(M,\Pi)$ is $G$-unimodular if and only if $c_{\mathcal{F}}^G \in \mathrm{H}_G^1(\mathcal{F})$ vanishes.
\end{theorem}

\begin{proof}
    Let us first recall that $(M,\Pi)$ is $G$-unimodular if and only if there exists a $G$-invariant volume form $\Theta\in\Omega^{2n+1}(M)$ such that its modular vector field vanishes in $G$-Poisson cohomology. We also recall the result of Theorem 9 in~\cite{guillemin_codimension_2011}: for the volume form $\Theta=\alpha\wedge\beta^n$, the modular vector field $X^\Theta$ is uniquely determined by
    \[
        \iota_{X^\Theta}j\beta=-j\nu,
    \]
    where $\nu\in\Omega^1(M)$ represents the normal variation of $\alpha$ defined in~\eqref{eq:normal-var-alpha}.

    Assume first that $(M,\Pi)$ is $G$-unimodular. Then there exists $f\in \mathcal{C}^\infty(M)^G$ such that $X^\Theta = \diff_{\Pi,G}f = \diff_\Pi f = X_f$. By~\eqref{eq:anchor-maps-reg-poisson},
    \[
        X^\Theta = \diff_\Pi f = i \big( \diff_{\Pi_{\mathcal F}}f \big) = -i \big( \beta^\sharp \diff_{\mathcal F} f \big).
    \]
    Therefore,
    \[
        -j\nu = \iota_{X^\Theta} j\beta = \beta^\flat(X^\Theta) = -\diff_{\mathcal F} f,
    \]
    and hence $j\nu = \diff_{\mathcal F} f$. It follows that $c_{\mathcal F}^G = [j\nu] = 0$.

    Conversely, suppose that $[j\nu] = 0$. Then there exists $f\in \mathcal{C}^\infty(M)^G$ such that $j\nu = \diff_{\mathcal F} f$. Reversing the previous computation, we obtain
    \[
        \iota_{X^\Theta} j\beta = -j\nu
        = -\diff_{\mathcal F} f
        = \iota_{X_f} j\beta.
    \]
    By uniqueness, $X^\Theta = X_f = \diff_{\Pi,G} f$. Thus the modular class vanishes in $G$-Poisson cohomology, and $(M,\Pi)$ is $G$-unimodular.
\end{proof}



\section{Non-equivariant cohomology theories for compact cosymplectic manifolds} \label{sec:cohomology-cosymplectic}

Let $(M, \alpha, \beta)$ be a compact cosymplectic with $\dim M = 2n + 1$. If a leaf $L$ of the foliation $\mathcal{F} = \ker \alpha$ is compact, a result of Guillemin, Miranda, and Pires, building on previous characterizations of Tischler~\cite[Thm.\,1]{tischler_fibering_1970} and Li~\cite[Thm.\,1]{li_topology_2008}, describes $M$ as a symplectic fibration over $\mathbb{S}^1$. Concretely,

\begin{theorem}[{Guillemin--Miranda--Pires \cite[Thm.\,13]{guillemin_codimension_2011}}] \label{thm:mapping-torus}
    Let $(M,\alpha,\beta)$ be a compact cosymplectic manifold with a compact leaf $L \subset M$. Then, all leaves are compact and $M$ is diffeomorphic to the \emph{mapping torus} of a symplectomorphism
    \begin{equation*}
        M \cong \frac{L\times[0,c]}{(x,0)\sim(\varphi(x),c)} \eqqcolon M_\varphi,
    \end{equation*}
    where $\varphi \coloneqq \Phi_R^c \colon (L,\beta_L) \to (L,\beta_L)$ is a symplectomorphism. In particular, $M$ fibres over $\mathbb{S}^1$ with fibre $L$ and monodromy $\varphi$.
\end{theorem}

The number $c > 0$ in the previous result is the minimum time $t > 0$ such that the flow $\Phi_R^t$ of the Reeb field maps $L$ to itself. It is called the \emph{modular period} of $(M, \alpha, \beta)$, and it is an invariant of the cosymplectic structure. We may sometimes denote it by $c \coloneqq \operatorname{mp}(\alpha, \beta)$.

\begin{remark}
    Observe that this theorem only treats the case in which 
    $L$ is compact. However, one can prove that there exists a sequence of cosymplectic manifolds without compact leaves converging to a cosymplectic manifold with compact leaves. This can be achieved using the approximation technique underlying Tischler’s theorem, and it fits within the scope of Sacksteder’s theorem on foliations \cite[Thm.\,2]{sacksteder-propertiesfoliations-1964}.
\end{remark}

An analogous description was used in \cite[Thm.\,4.3]{bazzoni_structure_2014} to establish additional restrictions on the cohomology groups of a manifold admitting a co-Kähler structure, yielding further obstructions to the existence of such a structure.

Let us now take $(M, \alpha, \beta)$ and identify it, following \Cref{thm:mapping-torus}, with a mapping torus $M_\varphi$. Let $(L,\beta_L)$ be a fibre and $\varphi$ the monodromy: this map is defined up to isotopy, and hence acts on cohomology unambiguously. Assume a compact $G$ acts by cosymplectomorphisms and $\varphi$ is $G$-equivariant.

On such a cosymplectic manifold there are several cohomology theories one might consider. In addition to the classical de Rham cohomology, one can consider the foliated cohomology groups for the foliation $\mathcal{F} = \ker \alpha$. Moreover, given a cosymplectic structure induces a regular symplectic foliation, and hence a Poisson structure as in  \Cref{ssec:poisson-geometry}, we can additionally consider the Poisson cohomology groups. All the previous cohomology theories have their natural equivariant counterparts: as such, we can consider the equivariant de Rham, foliated, and Poisson cohomologies.

The goal of this section is to provide a description of all the previous cohomology groups in terms of the fibration $\pi \colon M \to \mathbb{S}^1$. This approach is analogous to that of Bazzoni and Oprea \cite{bazzoni_structure_2014} to find restrictions in the cohomology groups of compact co-Kähler manifolds. The presentation will consequently include a mixture of both new and known results.

\subsection{De Rham cohomology of cosymplectic manifolds}

As noted at the beginning of the section, the approach we follow to the computation of the de Rham cohomology groups is by no means new. In fact, it is simply a restatement of the classical Wang long exact sequence in cohomology.

\begin{theorem} \label{thm:deRham-cosymplectic}
    Let $(M, \alpha, \beta)$ be a cosymplectic manifold described by a fibre bundle $\pi \colon M \to \mathbb{S}^1$ with fibre $L$ with monodromy map $\varphi \colon L \to L$, which is well-defined up to homotopy. Then, we have
    \begin{equation} \label{eq:Wang-deRham}
        \mathrm{H}^\bullet(M) \cong \ker ( \varphi^\bullet - \operatorname{id}_{\mathrm{H}^\bullet(L)} ) \oplus \operatorname{coker} ( \varphi^{\bullet - 1} - \operatorname{id}_{\mathrm{H}^{\bullet - 1}(L)} ).
    \end{equation}
\end{theorem}

\begin{remark}
    We have chosen the decomposition in terms of cokernels because it appears naturally from the long exact sequence in cohomology. However, one may reformulate equation \eqref{eq:Wang-deRham} as
    \begin{equation}
        \mathrm{H}^\bullet(M) \cong \ker (\varphi^\bullet - \operatorname{id}_{\mathrm{H}^{\bullet}(L)}) \oplus \ker (\varphi^{\bullet - 1} - \operatorname{id}_{\mathrm{H}^{\bullet - 1}(L)})
    \end{equation}
    since, for an endomorphism $\psi \colon V \to V$ of a vector space, the natural composition $\ker \psi \to V \to V / \operatorname{im} \psi = \operatorname{coker} \psi$ is an isomorphism.
\end{remark}

Given that we will apply a similar proof strategy throughout the remainder of the section, let us first recall the main steps of the argument. The proof proceeds by a Mayer--Vietoris argument. Choose an open cover $\{U, V\}$ of $\mathbb{S}^1$ by trivializing sets and assume they are contractible. Let $\psi_U$ and $\psi_V$ be the respective trivializations. We may choose the isomorphisms so that $i_{U \cap V}^U = \operatorname{id}_F$ and $i_{U \cap V}^V = \varphi$.

With respect to this open cover, we obtain the classical Mayer--Vietoris sequence:
\begin{equation} \label{eq:isom-ses-forms}
    \begin{tikzcd}[column sep=small]
        0 & {\Omega^\bullet(M) } & {\Omega^\bullet(\pi^{-1} U) \oplus \Omega^\bullet(\pi^{-1}V)} & {\Omega^\bullet( \pi^{-1} (U \cap V) )} & 0
        \arrow[from=1-1, to=1-2]
        \arrow[from=1-2, to=1-3]
        \arrow[from=1-3, to=1-4]
        \arrow[from=1-4, to=1-5]
    \end{tikzcd}
\end{equation}
Moreover, since $\pi^{-1}(U) \cong U \times L$ and $U$ is contractible we have $\mathrm{H}^{\bullet}(\pi^{-1}U) \cong \mathrm{H}^\bullet(L)$. Similarly, $\mathrm{H}^\bullet(\pi^{-1} V) \cong \mathrm{H}^\bullet(L)$. The long exact sequence in cohomology now reads
\begin{equation*}
    \begin{tikzcd}[column sep = large, cells={nodes={text height=2ex,text depth=0.75ex}}]
        \mathrm{H}^{k+1}(M) \arrow{r}{i_U^{k + 1} \oplus i_V^{k + 1}} & \cdots \\
        \mathrm{H}^k (M) \arrow{r}{i_U^{k} \oplus i_V^{k}} & \mathrm{H}^k (L) \oplus \mathrm{H}^k (L) \arrow{r}{(\varphi^*)^k - \operatorname{id}}
        \arrow[draw=none]{u}[name=Y, shape=coordinate]{}
        \arrow[draw=none]{d}[name=Z,shape=coordinate]{}
        & \mathrm{H}^k (L) \arrow[curarrow=Y]{ull}{} \\
        & \cdots \arrow{r}{(\varphi^*)^{k - 1} - \operatorname{id}} & \mathrm{H}^{k - 1} (L)
        \arrow[curarrow=Z]{ull}{}
    \end{tikzcd}
\end{equation*}
To complete the claim, one simply has to compute $\mathrm{H}^k(M)$ using the fact that the sequence is exact.

\subsection{The foliated cohomology groups}

Given that the foliated cohomology groups also satisfy a Mayer--Vietoris principle, we can apply the previous technique to obtain a Wang-like expression. Before proceeding, observe that, since $\mathcal{F}$ is defined in terms of the fiber bundle as $\mathcal{F} = \ker \diff \pi$, in a trivializing chart we have an isomorphism of foliated manifolds
\[
\pi^{-1}(U) \cong U \times L .
\]
Here, the foliation on $U$ is by points, and the foliation on $L$ is the total foliation (i.e., $L$ is the unique leaf). We will use the following Künneth formula for foliated cohomology in the product case, which follows e.g. from Bertelson~\cite[Prop.\,3.1]{bertelson_remarks_2011}.

\begin{lemma} \label{lem:foliated-cohomology-product}
 Let $M = U \times L$ with $L$ compact and with $U \subseteq \mathbb{R}$ and consider the foliation $\mathcal{F} = \operatorname{pr}^*_2 \mathrm{T}L$. In these considerations, the foliated cohomology groups are given by
    \begin{equation*}
        \mathrm{H}^\bullet(\mathcal{F}) \cong \mathcal{C}^\infty(U) \otimes \mathrm{H}^\bullet(L).
    \end{equation*}
\end{lemma}

With this result at hand, we can now establish the following theorem.
\begin{theorem} \label{thm:foliated-wang}
    Under the previous assumptions, the foliated cohomology groups can be described as follows:
    \begin{equation}
        \mathrm{H}^\bullet(\mathcal{F}) = \ker \big( (i_{U \cap V}^U)^* \otimes \operatorname{id}_{\mathrm{H}^\bullet(L)} - (i_{U \cap V}^V)^* \otimes \varphi^\bullet \big). \label{eq:foliated-cohomology-wang}
    \end{equation}
\end{theorem}

\begin{proof}
    The proof is similar to that of \Cref{thm:deRham-cosymplectic}. To clarify the role of the trivializations, let us express the classical Mayer–Vietoris sequence in terms of the corresponding isomorphic sequence.
    \begin{equation} \label{eq:ses-foliated-forms-wang}
    \begin{tikzcd}[column sep=scriptsize]
    	0 & {\Omega^\bullet(\mathcal{F})} & {\Omega(\operatorname{pr}_2|_U^* \mathrm{T}L)^\bullet \oplus \Omega^\bullet(\operatorname{pr}_2|_V^* \mathrm{T}L)} & {\Omega^\bullet \big(\operatorname{pr}_2|_{U \cap V}^* \mathrm{T}L\big)} & 0 \\
    	0 & {\Omega^\bullet(\mathcal{F}) } & {\Omega^\bullet(\mathcal{F}_{\pi^{-1} U}) \oplus \Omega^\bullet(\mathcal{F}_{\pi^{-1}V})} & {\Omega^\bullet\big( \mathcal{F}_{\pi^{-1} (U \cap V)} \big)} & 0
    	\arrow[from=1-1, to=1-2]
    	\arrow[from=1-2, to=1-3]
    	\arrow["{\operatorname{id}}"{description}, from=1-2, to=2-2]
    	\arrow["{j^\bullet}", from=1-3, to=1-4]
    	\arrow[from=1-3, to=2-3]
    	\arrow[from=1-4, to=1-5]
    	\arrow[from=1-4, to=2-4]
    	\arrow[from=2-1, to=2-2]
    	\arrow[from=2-2, to=2-3]
    	\arrow[from=2-3, to=2-4]
    	\arrow[from=2-4, to=2-5]
    \end{tikzcd}
    \end{equation}
    where the restriction map in the trivialized sequence is given by $j^\bullet = (i_{U \cap V}^U \times \operatorname{id}_L)^\bullet - (i_{U \cap V}^V \times \varphi)^\bullet$. An application now of \Cref{lem:foliated-cohomology-product} together with this trivialization of the short exact sequence yields the associated long exact sequence
    \begin{equation*}
        \begin{tikzcd}[column sep=small, cells={nodes={text height=2ex,text depth=0.75ex}}]
            \mathrm{H}^{k+1} (\mathcal{F}) \arrow{r}{i^{k + 1}} & \cdots \\
            \mathrm{H}^k (\mathcal{F}) \arrow{r}{i^k} & \mathcal{C}^\infty(U) \otimes \mathrm{H}^k(L) \oplus \mathcal{C}^\infty(V) \otimes \mathrm{H}^k(L) \arrow{r}{j^k}
            \arrow[draw=none]{u}[name=Y, shape=coordinate]{}
            \arrow[draw=none]{d}[name=Z,shape=coordinate]{}
            & \mathcal{C}^\infty (U \cap V) \otimes \mathrm{H}^k(L) \arrow[curvararrow=Y]{ull}{} \\
            & \cdots \arrow{r}{j^{k - 1}} & \mathcal{C}^\infty (U \cap V) \otimes \mathrm{H}^{k - 1}(L)
            \arrow[curvararrow=Z]{ull}{}
        \end{tikzcd}
    \end{equation*}
    
    The proof now follows by exactness of the previous diagram. We can observe, however, that in the decomposition $\mathrm{H}^\bullet(\mathcal{F}) \cong \ker(i^\bullet) \oplus \operatorname{im}(i^\bullet)$, equation \eqref{eq:foliated-cohomology-wang} corresponds to the vanishing of the groups $\operatorname{ker}(i^\bullet)$ or, in other words, to $\delta = 0$.
    
    We can prove this fact by showing that the short exact sequence \eqref{eq:ses-foliated-forms-wang} admits a splitting which intertwines the differentials. Following the classical proof of exactness for the Mayer--Vietoris sequence (cf. \cite[Prop.\,2.3]{bott_differential_1982}), in the open cover $U, V \subset \mathbb{S}^1$ we consider a subordinated partition of unity $\rho_{U}, \rho_{V}$. The pullbacks $\pi^* \rho_{U}, \pi^* \rho_{V} \in \mathcal{C}^\infty(M)$ are a partition of unity subordinated to the open cover $\{\pi^{-1}(U), \pi^{-1}(V)\}$ of $M$. Furthermore, since they are constant on the leaves of $\mathcal{F}$, we have $\diff_{\mathcal{F}} (\pi^* \rho_{U}) = \diff_{\mathcal{F}} (\pi^* \rho_{V}) = 0$. Consequently, the classical splitting of the short exact sequence \eqref{eq:ses-foliated-forms-wang} intertwines the foliated differentials, showing $\delta = 0$ in the induced long exact sequence in cohomology.
\end{proof}

It is not true in general that the dependence on the coefficients decomposes as a tensor product.
\begin{equation*}
    \ker \big( (i_{U \cap V}^U)^* \otimes \operatorname{id}_{\mathrm{H}^\bullet(L)} - (i_{U \cap V}^V)^* \otimes \varphi^\bullet \big) \not\cong \mathcal{C}^\infty(\mathbb{S}^1) \otimes \mathrm{H}^\bullet(L).
\end{equation*}
One may, however, obtain a description of the foliated cohomology groups if one interprets the right hand side in the bundle picture
\begin{equation} \label{eq:bundle-sections-cohomology}
    \mathcal{C}^\infty(\mathbb{S}^1) \otimes \mathrm{H}^\bullet(L) \cong \Gamma(\mathbb{S}^1 \times \mathrm{H}^\bullet(L)).
\end{equation}
This identification holds because the cohomology groups $\mathrm{H}^\bullet(L)$ are finitely generated, given $L$ is compact.

We can now construct $\mathcal{H}^\bullet(L) \to \mathbb{S}^1$ as the vector bundle over $\mathbb{S}^1$ with typical fiber $\mathrm{H}^\bullet(L)$ and monodromy map $\varphi^\bullet$. By the very construction of $\mathcal{H}^\bullet(L)$, we may assume that the vector bundle is trivial over the sets $U$ and $V$, and that, moreover, these isomorphisms satisfy $i_{U \cap V}^U \cong \operatorname{id}_{\mathrm{H}^\bullet(L)}$ and $i_{U \cap V}^V \cong \varphi^\bullet$. The sheaf condition at the level of sections yields the short exact sequence
\[\begin{tikzcd}[column sep=scriptsize]
    0 & {\Gamma(\mathcal{H}^\bullet(L))} & {\Gamma_U(\mathcal{H}^\bullet(L)) \oplus \Gamma_V(\mathcal{H}^\bullet(L))} & {\Gamma_{U \cap V}(\mathcal{H}^\bullet(L))} & 0
    \arrow[from=1-1, to=1-2]
    \arrow[from=1-2, to=1-3]
    \arrow[from=1-3, to=1-4]
    \arrow[from=1-4, to=1-5]
\end{tikzcd}\]
This result, together with the fact that the vector bundle is trivial over $U$ and $V$, and using a similar isomorphism to that in equation \eqref{eq:bundle-sections-cohomology}, yields an isomorphism of short exact sequences
\begin{equation*}
    \begin{tikzcd}[column sep=small]
        0 & {\Gamma\big(\mathcal{H}^\bullet(L)\big)} & {\mathcal{C}^\infty(U) \otimes \mathrm{H}^\bullet(L) \oplus \mathcal{C}^\infty(V) \otimes \mathrm{H}^\bullet(L) } & {\mathcal{C}^\infty(U \cap V) \otimes \mathrm{H}^\bullet(L)} & 0 \\
        0 & {\Gamma\big(\mathcal{H}^\bullet(L)\big)} & {\Gamma_U \big(\mathcal{H}^\bullet(L)\big) \oplus \Gamma_V \big(\mathcal{H}^\bullet(L)\big)} & {\Gamma_{U \cap V}\big(\mathcal{H}^\bullet(L)\big)} & 0
        \arrow[from=1-1, to=1-2]
        \arrow[from=1-2, to=1-3]
        \arrow[from=1-2, to=2-2]
        \arrow[from=1-3, to=1-4]
        \arrow[from=1-3, to=2-3]
        \arrow[from=1-4, to=1-5]
        \arrow[from=1-4, to=2-4]
        \arrow[from=2-1, to=2-2]
        \arrow[from=2-2, to=2-3]
        \arrow[from=2-3, to=2-4]
        \arrow[from=2-4, to=2-5]
    \end{tikzcd}
\end{equation*}
With these computations we have obtained the description of the foliated cohomology groups in terms of the following

\begin{theorem} \label{thm:foliated-cohomology-sections-bundle}
    For a fibration $\pi \colon M \to \mathbb{S}^1$ with compact fibers isomorphic to $L$, the foliated cohomology groups can be computed as
    \begin{equation} \label{eq:foliated-cohomology-sections-bundle}
        \mathrm{H}^\bullet(\mathcal{F}) \cong \Gamma \mathcal{H}^\bullet(L),
    \end{equation}
    where $\mathcal{H}^\bullet(L) \to \mathbb{S}^1$ is the vector bundle with fibre $\mathrm{H}^\bullet(L)$ and monodromy $\varphi^\bullet \colon \mathrm{H}^\bullet(L) \to \mathrm{H}^\bullet(L)$.
\end{theorem}

\subsection{Poisson cohomology of cosymplectic manifolds}

The Poisson cohomology groups of a cosymplectic manifold were computed by Osorno \cite[Thm.\,3.2.17]{osorno} using a long exact sequence in cohomology for codimension-one regular Poisson manifolds. Let us reproduce the statement.

\begin{theorem}[{Osorno~\cite[Thm.\,3.2.17]{osorno}}] \label{thm:poisson-cohomology-cosymplectic}
    Let $(M^{2n+1},\alpha,\beta)$ be a cosymplectic manifold and $\Pi$ the associated Poisson structure. Then the Poisson cohomology of $(M,\Pi)$ decomposes as
    \begin{equation}
        \mathrm{H}^\bullet_\Pi(M) \cong \mathrm{H}^\bullet( \mathcal{F} ) \oplus \mathrm{H}^{\bullet - 1}( \mathcal{F} ).
    \end{equation}
    where $\mathcal{F}$ is the characteristic foliation determined by $\ker \alpha$.  
\end{theorem}

\begin{remark}
   A similar result in the setting of cosymplectic manifolds had previously been obtained by Fernández, Ibáñez, and de León~\cite[Cor.\,3.14]{Fernandez_canonical_1998} for the Poisson \emph{homology} groups. They show that there exists an isomorphism
    \begin{equation}
        \mathrm{H}_k(M) \cong \mathrm{H}^{2n - k}(\mathcal{F}) \oplus \mathrm{H}^{2n - k + 1}(\mathcal{F}).
    \end{equation}
\end{remark}

While the approach using a long exact sequence in cohomology is conceptually clear, for the eventual computation of Poisson cohomology in the equivariant setting it will be more convenient to have a working description of this isomorphism. We devote the remainder of the section to this task.

\begin{proof}
    Let us recall the splitting of the short exact sequence \eqref{eq:ses-tangent-foliation} induced by the Reeb field $R$. At the level of multivector fields, this splitting induces the isomorphism
    \begin{equation} \label{eq:splitting-reeb-sections}
        \begin{array}{rccc}
             \kappa \colon & \mathfrak{X}^\bullet(\mathcal{F}) \oplus \mathfrak{X}^{\bullet - 1}(\mathcal{F}) & \longrightarrow & \mathfrak{X}^\bullet(M) \\
             & (X, Y) & \longmapsto & i(X) + R \wedge i(Y).
        \end{array}
    \end{equation}
    Moreover, it is easy to see that the isomorphism induces an isomorphism (up to a sign) at the level of cochain complexes, since the Reeb vector field is Poisson:
    \begin{align*}
        \diff_\Pi \kappa(X, Y) &= \diff_\Pi i(X) - R \wedge \diff_\Pi i(Y) \\
        &= i(\diff_{\Pi_{\mathcal{F}}} X) + R \wedge i(- \diff_{\Pi_{\mathcal{F}}} Y) = \kappa(\diff_{\Pi_{\mathcal{F}}} X, - \diff_{\Pi_{\mathcal{F}}} Y).
    \end{align*}
    To conclude the proof, simply observe that the Poisson tensor $\Pi_{\mathcal{F}}$ is actually symplectic on the bundle $\mathrm{T} \mathcal{F}$ and, consequently, we have an isomorphism $\mathrm{H}_{\Pi_{\mathcal{F}}}^\bullet (\mathcal{F}) \cong \mathrm{H}^\bullet(\mathcal{F})$. Combining the previous observations we conclude the statement of the theorem.
\end{proof}

\begin{remark}[Geometric meaning of the decomposition]\label{rmk:dec}
    The foliation $\mathcal{F}_\Pi$ integrates the kernel of~$\alpha$ and is
    regular of codimension~$1$. Each leaf is a symplectic manifold with symplectic
    form $\omega|_{\mathcal{F}_\Pi}$. From the point of view of Poisson geometry,
    the decomposition above reflects two fundamental facts:
    
    \begin{enumerate}
        \item The complex computing Poisson cohomology splits into
        ``tangential'' and ``transversal'' components relative to the foliation $\mathcal{F}_\Pi$.
        
        \item The Poisson differential contains, besides the usual leafwise de~Rham
        differential, an additional operator detecting the transverse geometry encoded
        by $\alpha$.
    \end{enumerate}
    This results in a shifted direct sum of two copies of the foliated de~Rham cohomology.
\end{remark}

The following corollary was proved in \cite[Cor.\,2]{martinezmiranda}:
\begin{corollary}[Top Poisson cohomology]\label{cor:topPoisson}
    For a cosymplectic manifold $(M^{2n+1},\alpha,\beta)$ we have
    \[
        \mathrm{H}_\Pi^{2n+1}(M) \cong \mathrm{H}^{2n}(\mathcal{F}_\Pi).
    \]
\end{corollary}

\begin{remark}[Relation with the symplectic case]
    If $(M,\omega)$ is a symplectic manifold (i.e., there is ``no foliation direction’’), then
    $\mathcal{F}_\Pi$ has a single leaf, and
    \[
        \mathrm{H}_\Pi^\bullet(M) \cong \mathrm{H}^\bullet(M)
    \]
    recovering the standard fact that Poisson cohomology
    of a symplectic manifold identifies with its de~Rham cohomology.
    Thus, the cosymplectic case constitutes the ``next simplest’’ situation
    after the symplectic case and provides a clean model for Poisson cohomology
    in corank-one settings.
\end{remark}

\begin{remark}
    Following the previous sections, one may try to approach the computation from the Mayer--Vietoris perspective in the compact setting. In this case, the local structure of the cohomology groups is given by $\mathrm{H}_{\Pi}^\bullet(U \times L) \cong \mathfrak{X}^\bullet(U) \otimes \mathrm{H}^\bullet(L)$. Following this idea, the fact that we obtain two copies of $\mathrm{H}^\bullet(\mathcal{F})$ follows from the isomorphism $\mathfrak{X}^\bullet(U) \cong \mathcal{C}^\infty(U) \oplus \mathcal{C}^\infty(U) \, \partial_t$, and the monodromy map $\varphi$ appears implicitly in the groups $\mathrm{H}^\bullet(\mathcal{F})$ following equation \eqref{eq:foliated-cohomology-wang}.
\end{remark}

\section{Equivariant cohomology theories for compact cosymplectic structures} \label{sec:equiv-cohomology-cosymplectic}

In this section we discuss the equivariant counterparts of the previous cohomology theories. Let us consequently endow $M$ with an equivariant cosymplectic structure $(\alpha, \beta)$ as in \Cref{def:equivariant-cosymplectic}. Following \Cref{ssec:equiv-coorient-codimone-fol}, we will only consider group actions which preserve the leaves of the foliation $\mathcal{F}_\alpha$ individually: in terms of the fibration $\pi \colon M \to \mathbb{S}^1$, this is equivalent to $\pi \rho_g = \pi$ for all $g \in G$.

Given our proofs will proceed by applying the Mayer--Vietoris principle in suitable trivializations, we first have to find local normal models for the $G$-action. In general, the $G$-action may vary from fibre to fibre. In our case, the existence of a cosymplectic structure imposes additional constraints. We have the following result.

\begin{lemma} \label{lem:trivialization-action}
    Let $(M, \alpha, \beta)$ be a cosymplectic manifold in the assumptions of \Cref{thm:mapping-torus}, i.e., assume the foliation $\mathcal{F}$ is defined in terms of a fibration $\pi \colon M \to \mathbb{S}^1$. Let $\rho \colon G \times M \to M$ be a $G$-action such that $\pi \rho_g = \pi$ for all $g \in G$.

    Then, for every $p \in \mathbb{S}^1$ there exists a trivialization $\varphi \colon \pi^{-1}(U) \to U \times L$ and a $G$-action $\gamma' \colon G \times L \to L$ such that, denoting $\gamma_ g \coloneqq \varphi \rho_g \varphi^{-1}$, we have $\operatorname{pr}_2 \gamma_g =  \gamma'_g \operatorname{pr}_2$ for all $g \in G$. In other words, the induced action satisfies
    \begin{equation*}
        \gamma_g(t, q) = \big( t, \gamma_g'(q) \big).
    \end{equation*}
\end{lemma}

We sketch the details of the proof. Remember from \cite{guillemin_codimension_2011} that the trivializing charts of the fibration $\pi \colon M \to \mathbb{S}^1$ are obtained by flowing a compact leaf $L$ along the Reeb vector field $R$ of the cosymplectic structure, i.e.,
\begin{equation*}
    \begin{array}{rccc}
        \varphi \colon & (-\varepsilon, \varepsilon) \times L & \longrightarrow & \pi^{-1}(-\varepsilon, \varepsilon) \\
         & (t, p) & \longmapsto & \Phi_R^t(p).
    \end{array}
\end{equation*}
Because $R$ is $G$-invariant, the flow $\Phi^t_R$ conjugates the $G$-action.

In what follows, the model action of $G$ on the fibre $L$ will be of great importance in all computations. More concretely, many of our results will be expressed in terms of the equivariant cohomology groups $\mathrm{H}_G^\bullet(L)$. Remember that $L$ is a symplectic manifold and that, for an equivariant cosymplectic structure, the $G$-action on $L$ is Hamiltonian (cf. \cite[Thm.\,3.2]{miranda_hamiltonian_2026}). It is a classical theorem of Kirwan that compact Hamiltonian $G$-spaces are \emph{equivariantly formal}. More precisely:

\begin{theorem}[Kirwan {\cite{kirwan}}] \label{thm:kirwan-formality}
    Let $(M,\omega)$ be a compact symplectic manifold equipped with a Hamiltonian action of a compact, connected Lie group $G$. Then the $G$-equivariant cohomology of $M$ is
    equivariantly formal, that is,
    \begin{equation} \label{eq:kirwan-formality}
        \mathrm{H}_G^\bullet(M;\mathbb{R}) \cong \mathrm{H}^\bullet(M;\mathbb{R}) \otimes \mathrm{H}^\bullet(\mathrm{B}G; \mathbb{R})
    \end{equation}
    as graded $\mathbb{R}$-algebras. Equivalently, the Serre spectral sequence of the Borel fibration
    \begin{equation} \label{eq:borel-fibration}
        M \hookrightarrow M_G \rightarrow \mathrm{B}G
    \end{equation}
    collapses at the $E_2$-page.
\end{theorem}

Recall the equivariant cohomology of a $G$-manifold $M$ may be computed either as $\mathrm{H}^\bullet(M_G)$ or via the Cartan model $(\Omega_G^\bullet(M), \diff_G)$. The Serre spectral sequence of the Borel fibration \eqref{eq:borel-fibration} induces on $\mathrm{H}_G^\bullet(M)$ the \emph{Borel filtration}, defined by the skeletal filtration of $\mathrm{B}G$.  Filtering the Cartan complex by polynomial degree in $\operatorname{Sym}^\bullet(\mathfrak{g}^\ast)$ yields a second spectral sequence, which is used (cf. Tu and Arabia~\cite[Appendix~A]{tu_introductory_2020}) to prove the equivariant de~Rham isomorphism
\begin{equation*}
    \mathrm{H}^\bullet(M_G) \cong \mathrm{H}(\Omega_G^\bullet(M), \diff_G)
\end{equation*}
via successive approximations $M_G(m) = M \times \mathrm{E}G(m)$ to the homotopy quotient $M_G$. Consequently, the spectral sequence by polynomial degree also computes the equivariant cohomology of $M$.

\begin{remark}
    Observe in the case where equivariant formality holds that, since $\operatorname{Sym}^\bullet(\mathfrak{g}^*)^G$ is finite-dimensional at every fixed degree and, similarly, so is $\mathrm{H}^\bullet(L)$ since $L$ is compact, the equivariant cohomology groups $\mathrm{H}_G^\bullet(L)$ are finite-dimensional for every degree.

    If the manifold is not equivariantly formal, the same remark applies observing the groups $\mathrm{H}^\bullet(M) \otimes \mathrm{H}^\bullet(\mathrm{B}G)$ correspond to the second page of the spectral sequence of the fibration $M_G \to \mathrm{B}G$ which computes $\mathrm{H}_G^\bullet(M)$.
\end{remark}

\subsection{Equivariant de Rham cohomology}

To compute the equivariant de Rham cohomology groups, we apply the same technique as in the previous sections. The necessary technical lemma, i.e., the Mayer--Vietoris principle, holds for equivariant cohomology. For a proof we refer to \cite[Thm.\,25.15]{tu_introductory_2020}.

Let us now discuss the local structure of equivariant de~Rham cohomology, as a preliminary step to the application of the Mayer--Vietoris theorem.

\begin{lemma} \label{lem:equivariant-de-rham-product}
    Let us consider a product manifold $U \times L$ equipped with an action $\rho$ of a Lie group $G$ satisfying $\operatorname{pr}_1 \rho_g = \operatorname{pr}_1$ for all $g \in G$. Then, the equivariant cohomology groups are given by
    \begin{equation}
        \mathrm{H}_G^\bullet(U \times L) \cong \mathrm{H}_G^\bullet(L).
    \end{equation}
\end{lemma}

\begin{proof}
    Because $U \subset \mathbb{R}$ is contractible and the $G$-action on $U$ is trivial, we have an equivariant homotopy between $U \times L$ and $L$. This yields a homotopy equivalence between the homotopy quotients $(U \times L) \times_G \mathrm{E}G$ and $L \times_G \mathrm{E}G$, proving that $\mathrm{H}_G^\bullet(U \times L) \cong \mathrm{H}^\bullet_G(L)$.
\end{proof}

\begin{theorem} \label{thm:mv-equiv-deRham}
    Let $M$ be a manifold with an action of a compact, connected Lie group $G$ and let $(\underline{\alpha}, \underline{\beta})$ be an equivariant cosymplectic structure. Moreover, assume $M$ can be described as a fibre bundle $\pi \colon M \to \mathbb{S}^1$ following \Cref{thm:mapping-torus}. Then, the equivariant cohomology groups can be computed as
    \begin{equation} \label{eq:mv-equiv-deRham}
        \mathrm{H}^\bullet_G(M) \cong \ker \big( (\varphi^\bullet - \operatorname{id}_{\mathrm{H}^\bullet(L)}) \otimes \operatorname{id}_{\operatorname{Sym}^\bullet(\mathfrak{g}^*)} \big) \oplus \operatorname{coker} \big( (\varphi^{\bullet - 1} - \operatorname{id}_{\mathrm{H}^{\bullet - 1}(L)}) \otimes \operatorname{id}_{\operatorname{Sym}^\bullet(\mathfrak{g}^*)} \big).
    \end{equation}
\end{theorem}

\begin{proof}
    For the proof we use the approach to Wang's sequence through a Mayer--Vietoris argument. Let us choose a covering of the circle $\mathbb{S}^1$ by two open sets $U, V \subset \mathbb{S}^1$ such that $U \cong V \cong \mathbb{R}$ and $U \cap V \cong \mathbb{R} \sqcup \mathbb{R}$, which we can furthermore assume to be trivializing sets in the assumptions of \Cref{lem:trivialization-action}. We have an induced covering of the total space by the sets $\pi^{-1}(U)$ and $\pi^{-1}(V)$. Let us choose trivializations of the fibre bundle $\pi \colon M \to \mathbb{S}^1$, which we denote by $\psi_U \colon \pi^{-1}(U) \to U \times L$, $\psi_V \colon \pi^{-1}(V) \to V \times L$, and $\psi_{U \cap V} \colon \pi^{-1}(U \cap V) \to (U \cap V) \times L$. Furthermore we can assume that the restriction maps in these trivializations are given by
    \begin{equation*}
        \psi_U \circ i_{\pi^{-1}(U \cap V)}^{\pi^{-1}(U)} \circ \psi_{U \cap V}^{-1} = \operatorname{id}_{U \cap V}^U \times \operatorname{id}_L \quad \text{and} \quad \psi_V \circ i_{\pi^{-1}(U \cap V)}^{\pi^{-1}(V)} \circ \psi_{U \cap V}^{-1} = \operatorname{id}_{U \cap V}^V \times \varphi,
    \end{equation*}
    where $\varphi \colon L \to L$ is the monodromy map of the bundle.

    Now, because we have $\pi \rho_g = \pi$ by assumption, the sets $\pi^{-1}(U)$ and $\pi^{-1}(V)$ are $G$-invariant. The equivariant Mayer--Vietoris theorem (cf. \cite[Thm.\,25.15]{tu_introductory_2020}) implies the following short sequence is exact:
    \begin{equation}
        \begin{tikzcd}[column sep=small]
        	0 & {\Omega_G^\bullet(M) } & {\Omega_G^\bullet(\pi^{-1} U) \oplus \Omega_G^\bullet(\pi^{-1}V)} & {\Omega_G^\bullet( \pi^{-1} (U \cap V) )} & 0
        	\arrow[from=1-1, to=1-2]
        	\arrow[from=1-2, to=1-3]
        	\arrow[from=1-3, to=1-4]
        	\arrow[from=1-4, to=1-5]
        \end{tikzcd}
    \end{equation}
    The corresponding long exact sequence, together with \Cref{lem:equivariant-de-rham-product}, implies the long exact sequence in cohomology
    \begin{equation*}
        \begin{tikzcd}[, cells={nodes={text height=2ex,text depth=0.75ex}}]
           \mathrm{H}^{k+1}_G (M) \arrow{r} & \cdots \\
           \mathrm{H}^k_G (M) \arrow{r} & \mathrm{H}_G^k(L) \oplus \mathrm{H}_G^k(L) \arrow{r}{j^k}
          \arrow[draw=none]{u}[name=Y, shape=coordinate]{}
          \arrow[draw=none]{d}[name=Z,shape=coordinate]{}
          & \mathrm{H}_G^k(L) \arrow[curarrow=Y]{ull}{} \\
            & \cdots \arrow{r}{j^{k - 1}} & \mathrm{H}_G^{k - 1}(L)
          \arrow[curarrow=Z]{ull}{}
        \end{tikzcd}
    \end{equation*}
    where the morphisms $j^\bullet$ are explicitly given by $j^\bullet = (\varphi^\bullet -\operatorname{id}_{\mathrm{H}^\bullet(L)}) \otimes \operatorname{id}_{\operatorname{Sym^\bullet(\mathfrak{g}^*)}}$. A standard argument using the exactness of the diagram concludes the proof.
\end{proof}

Observe that, in the case where we have a trivial $\mathbb{S}^1$-bundle, the equivariant Wang sequence specifies to an equivariant version of the Künneth formula.

\begin{corollary} \label{cor:equiv-Kunneth}
    Let $L \times \mathbb{S}^1$ be a smooth manifold with $L$ compact and equipped with a $G$-action on $L$ and consider its trivial extension to the $\mathbb{S}^1$-factor. Then, the equivariant cohomology group satisfies the relation
    \begin{equation} \label{eq:equiv-kunneth}
        \mathrm{H}_G^\bullet(L \times \mathbb{S}^1) \cong \mathrm{H}_G^\bullet(L) \otimes \mathrm{H}^\bullet(\mathbb{S}^1).
    \end{equation}
\end{corollary}

Equation \eqref{eq:mv-equiv-deRham} completely characterizes the equivariant cohomology groups of $M$ in terms of those for the fibre $L$. However, more can be said in the setting of equivariant cosymplectic structures: since the $G$-action is Hamiltonian when restricted to $L$ (cf.~\cite[Thm.\,3.2]{miranda_hamiltonian_2026}), Kirwan's \Cref{thm:kirwan-formality} applies and we obtain the following refinement of the equivariant cohomology groups.

\begin{corollary} \label{thm:dR-equiv-formality}
    In assumptions of \Cref{thm:mv-equiv-deRham}, the equivariant cohomology groups are given by
    \begin{equation}
        \mathrm{H}_G^\bullet(M) \cong \mathrm{H}^\bullet(M) \otimes \operatorname{Sym}^\bullet(\mathfrak{g}^*)^G.
    \end{equation}
    In particular, $M$ is equivariantly formal.
\end{corollary}

\begin{proof}
    From equation \eqref{eq:kirwan-formality} we have $\mathrm{H}^\bullet_G(L) \cong \mathrm{H}^\bullet(L) \otimes \operatorname{Sym}(\mathfrak{g}^*)^G$ and, consequently, $\ker ( (\varphi^\bullet - \operatorname{id}_{\mathrm{H}^\bullet(L)}) \otimes \operatorname{id}_{\operatorname{Sym}^\bullet(\mathfrak{g}^*)} ) = \ker(\varphi^\bullet - \operatorname{id}_{\mathrm{H}^\bullet(L)}) \otimes \operatorname{Sym}^\bullet(\mathfrak{g}^*)^G$, as can be checked by expanding a general element in a basis. For a similar reason, $\operatorname{coker}((\varphi^\bullet - \operatorname{id}_{\mathrm{H}^\bullet(L)}) \otimes \operatorname{id}_{\operatorname{Sym}^\bullet(\mathfrak{g}^*)}) \cong \operatorname{coker}(\varphi^\bullet - \operatorname{id}_{\mathrm{H}^\bullet(L)}) \otimes \operatorname{Sym}^\bullet(\mathfrak{g}^*)^G$. Therefore, from this observation together with \Cref{thm:deRham-cosymplectic} and \Cref{thm:mv-equiv-deRham} we conclude
    \begin{align*}
        \mathrm{H}_G^\bullet(M) &\cong \ker (\varphi^\bullet - \operatorname{id}_{\mathrm{H}^\bullet(L)}) \otimes \operatorname{Sym}^\bullet(\mathfrak{g}^*)^G \oplus \operatorname{coker} (\varphi^{\bullet - 1} - \operatorname{id}_{\mathrm{H}^{\bullet - 1}(L)}) \otimes \operatorname{Sym}^\bullet(\mathfrak{g}^*)^G \\
        &\cong \big( \ker (\varphi^\bullet - \operatorname{id}_{\mathrm{H}^\bullet(L)}) \oplus \operatorname{coker} (\varphi^{\bullet - 1} - \operatorname{id}_{\mathrm{H}^{\bullet - 1}(L)}) \big) \otimes \operatorname{Sym}^\bullet(\mathfrak{g})^G \\
        &\cong \mathrm{H}^\bullet(M) \otimes \operatorname{Sym}^\bullet(\mathfrak{g}^*)^G. \qedhere
    \end{align*}
\end{proof}

\subsection{Equivariant foliated cohomology}

As in the previous section, the method of proof is by exactness of the Mayer--Vietoris sequence. Let us first state and proof the customary lemma regarding the corresponding cohomology of the local normal forms.

\begin{proposition} \label{prop:foliated-equiv-kunneth}
    Let us consider a product manifold $\mathbb{R} \times L$ with $L$ compact and with the foliation $\mathcal{F} = \operatorname{pr}_2^* \mathrm{T} L$. Assume a compact, connected Lie group $G$ acts on $L$ in Hamiltonian fashion and consider the induced action on $\mathbb{R} \times L$.

    Under these assumptions, the equivariant foliated cohomology groups can be computed as
    \begin{equation} \label{eq:local-kunneth-equivariant-foliated}
        \mathrm{H}_G^\bullet(\mathcal{F}) = \mathcal{C}^\infty(\mathbb{R}) \otimes \mathrm{H}^\bullet(L) \otimes \operatorname{Sym}^\bullet(\mathfrak{g}^*)^G.
    \end{equation}
\end{proposition}

\begin{proof}
    Let us begin the proof by constructing a splitting at the level of cochain complexes. We know from Bertelson~\cite[Prop.\,2.3]{bertelson_remarks_2011} that the canonical inclusion $\mathcal{C}^\infty(\mathbb{R}) \otimes \Omega(L) \xhookrightarrow{} \Omega^\bullet(\mathrm{pr}_2^* \mathrm{T}L)$ has dense image. Consequently, the canonical inclusion
    \begin{equation*}
        \mathcal{C}^\infty(\mathbb{R}) \otimes \Omega(L)  \otimes \operatorname{Sym}^\bullet(\mathfrak{g}^*) \xhookrightarrow{} \Omega^\bullet(\mathrm{pr}_2^* \mathrm{T}L) \otimes \operatorname{Sym}^\bullet(\mathfrak{g}^*)
    \end{equation*}
    has also dense image (for example, looking degree by degree and using the fact that $\operatorname{Sym}^k(\mathfrak{g}^*)$ is finitely generated for all $k$). Moreover, the natural inclusion commutes with the corresponding $G$-actions. Hence, after taking completions we have an isomorphism at the level of invariants
    \begin{equation*}
        \mathcal{C}^\infty(\mathbb{R}) \cotimes \Omega_G^\bullet(L) \cong \Omega_G^\bullet(\operatorname{pr}_2^* \mathrm{T}L).
    \end{equation*}

    Now, following Bertelson~\cite[Sec.\,3]{bertelson_remarks_2011}, it suffices to prove that $\mathrm{H}_G^\bullet(L)$ is Hausdorff, given we have already seen it is finite-dimensional. The proof now follows Bertelson~\cite[Sec.\,3]{bertelson_remarks_2011}. Observe we know the operator $\diff$ is continuous and so is $\iota_{X^\#}$ because $M$ is compact. Therefore, the equivariant differential $\diff_G = \diff \otimes \operatorname{id} - \sum_{i = 1}^m \iota_{X^\#_i} \otimes \sigma^i$ is continuous. Because the equivariant cohomology groups are finite-dimensional from equivariant formality \eqref{eq:kirwan-formality}, the same argument as in \cite[Rmk.\,3.2]{bertelson_remarks_2011} shows the groups $\mathrm{H}_G^\bullet(L)$ are Hausdorff. The proof of \cite[Prop.\,3.1]{bertelson_remarks_2011} applies once again, showing the Künneth formula holds in this case.
\end{proof}

\begin{theorem}
    Let $M$ be a compact smooth manifold endowed with the action of a compact, connected group $G$ and let $(\alpha, \beta - \mu)$ be an equivariant cosymplectic structure in $M$. Assume $M$ has a compact leaf $L$ and identify it with a fibration $\pi \colon M \to \mathbb{S}^1$.

    Then, the equivariant foliated cohomology groups can be computed as
    \begin{equation}
        \mathrm{H}^\bullet_G(\mathcal{F}) \cong \mathrm{H}^\bullet(\mathcal{F}) \otimes \operatorname{Sym}^\bullet(\mathfrak{g}^*)^G.
    \end{equation}
\end{theorem}

\begin{proof}
    This proof is similar to those of \Cref{thm:foliated-wang,thm:mv-equiv-deRham} and \Cref{thm:dR-equiv-formality}, and, as such, we only outline the main steps.
    
    Let us consider the open cover $\pi^{-1}(U), \pi^{-1}(V)$ of $M$ as in the proof of \Cref{thm:mv-equiv-deRham}. From the equivariant Mayer--Vietoris sequence we obtain a long exact sequence in cohomology which, together with  \Cref{prop:foliated-equiv-kunneth}, can be put in a similar form to that in the proof of \Cref{thm:foliated-wang}. The same argument implies now the relation
    \begin{equation*}
        \mathrm{H}^\bullet_G(\mathcal{F}) = \ker \big( (i_{U \cap V}^U)^* \otimes \operatorname{id}_{\mathrm{H}^\bullet(L)} \otimes \operatorname{id}_{\operatorname{Sym}^\bullet(g^*)} - (i_{U \cap V}^V)^* \otimes \varphi^\bullet \otimes \operatorname{id}_{\operatorname{Sym}^\bullet(g^*)} \big).
    \end{equation*}
    However, we can still interpret these elements as sections of a certain bundle $\mathcal{H}_G^\bullet(L) \to \mathbb{S}^1$ following \Cref{thm:foliated-cohomology-sections-bundle}, constructed with fibre $\mathrm{H}^\bullet_G(L) \cong \mathrm{H}^\bullet(L) \otimes \operatorname{Sym}^\bullet(\mathfrak{g}^*)^G$ and monodromy $(\operatorname{id}_{\mathrm{H}^\bullet(L)} - \varphi^\bullet) \otimes \operatorname{id}_{\operatorname{Sym}^\bullet(\mathfrak{g}^*)}$. As a direct consequence of equivariant formality we have $\mathcal{H}_G^\bullet(L) \cong \mathcal{H}^\bullet(L) \otimes \operatorname{Sym}^\bullet(\mathfrak{g}^*)^G$ and, together with equation \eqref{eq:foliated-cohomology-sections-bundle}, we consequently have
    \begin{equation*}
        \Gamma(\mathcal{H}_G^\bullet(L)) \cong \Gamma(\mathcal{H}^\bullet(L) \otimes \operatorname{Sym}^\bullet(\mathfrak{g}^*)^G) \cong \Gamma(\mathcal{H}^\bullet(L)) \otimes \operatorname{Sym}^\bullet(\mathfrak{g}^*)^G \cong \mathrm{H}^\bullet(\mathcal{F}) \otimes \operatorname{Sym}^\bullet(\mathfrak{g}^*)^G. \qedhere
    \end{equation*}
\end{proof}

\subsection{Equivariant Poisson cohomology}

We conclude this section by explicitly describing the equivariant Poisson cohomology groups of a cosymplectic manifold. The result is a verbatim analogue of the characterization \Cref{thm:poisson-cohomology-cosymplectic} in terms of the foliated cohomology groups. 

Before the proof, let us observe that the splitting defined in equation \eqref{eq:splitting-reeb-sections} is $G$-equivariant given the $G$-action descends to the foliation and the Reeb field $R$ is $G$-invariant. Similarly to the symplectic case, the anchor map $\omega^\flat \otimes \operatorname{id}_{\operatorname{Sym}^\bullet(\mathfrak{g}^*)}$ induces an isomorphism of differential complexes $\mathfrak{X}_{G}^\bullet(\mathcal{F}) \cong \Omega^\bullet_G(\mathcal{F})$ which induces an isomorphism in cohomology $\mathrm{H}_{\Pi, G}^\bullet(\mathcal{F}) \cong \mathrm{H}^\bullet_G(\mathcal{F})$.

\begin{theorem} \label{thm:equiv-poiss-splitting}
    Let $M$ be a smooth manifold together with an action of a compact, connected Lie group $G$ and let $(\alpha, \beta - \mu)$ be an equivariant cosymplectic structure. Assume $\Pi \in \mathfrak{X}^2(M)$ and $\Pi_{\mathcal{F}} \in \mathfrak{X}^2(\mathcal{F})$ are the Poisson structures defined in \Cref{ssec:poisson-geometry}.

    Then, the splitting $\kappa$ in \eqref{eq:splitting-reeb-sections} is $G$-equivariant and hence induces an isomorphism (up to a sign) of differential complexes
    \begin{equation} \label{eq:equiv-Poisson-splitting}
        \kappa \otimes \operatorname{id}_{\operatorname{Sym}^\bullet(\mathfrak{g}^*)} \colon \mathfrak{X}_G^\bullet(\mathcal{F}) \otimes \mathfrak{X}_G^{\bullet - 1}(\mathcal{F}) \stackrel{\sim}{\longrightarrow} \mathfrak{X}^\bullet_G(M).
    \end{equation}
    As a consequence, we obtain an isomorphism of the corresponding cohomology groups
    \begin{equation}
        \mathrm{H}_{\Pi, G}^\bullet(M) \cong \mathrm{H}^\bullet_G(\mathcal{F}) \oplus \mathrm{H}_G^{\bullet - 1}(\mathcal{F}).
    \end{equation}
\end{theorem}

\begin{proof}
    Let us check that the map \eqref{eq:equiv-Poisson-splitting} induces an isomorphism between the corresponding complexes. As the map $\kappa$ is an isomorphism (cf. the proof of \Cref{thm:poisson-cohomology-cosymplectic}), after taking tensor products we obtain an isomorphism
    \begin{equation*}
        \kappa \otimes \operatorname{id}_{\operatorname{Sym}^\bullet(\mathfrak{g}^*)} \colon \big( \mathfrak{X}^\bullet(\mathcal{F}) \oplus \mathfrak{X}^{\bullet - 1}(\mathcal{F}) \big) \otimes \operatorname{Sym}^\bullet(\mathfrak{g}^*) \to \mathfrak{X}^\bullet(M) \otimes \operatorname{Sym}^\bullet(\mathfrak{g}^*).
    \end{equation*}
    Since the map $\kappa$ is equivariant because the Reeb field $R \in \mathfrak{X}(M)$ is invariant, so is the map $\kappa \otimes \operatorname{id}_{\operatorname{Sym}^\bullet(\mathfrak{g}^*)}$. Observe that we have the natural isomorphism
    \begin{equation*}
        (\mathfrak{X}^\bullet(\mathcal{F}) \oplus \mathfrak{X}^{\bullet - 1}(\mathcal{F})) \otimes \operatorname{Sym}^\bullet(\mathfrak{g}^*) \cong \mathfrak{X}^\bullet(\mathcal{F}) \otimes \operatorname{Sym}^\bullet(\mathfrak{g}^*) \oplus \mathfrak{X}^{\bullet - 1}(\mathcal{F}) \otimes \operatorname{Sym}^\bullet(\mathfrak{g}^*),
    \end{equation*}
    which is clearly $G$-equivariant with respect to the diagonal action in the right hand side. As all the previous isomorphisms are equivariant, they descend to an isomorphism of invariant sets and hence to the isomorphism
    \begin{equation*}
        \kappa \otimes \operatorname{id}_{\operatorname{Sym}^\bullet(\mathfrak{g}^*)} \colon \mathfrak{X}^\bullet_G(\mathcal{F}) \otimes \mathfrak{X}_G^{\bullet - 1}(\mathcal{F}) \longrightarrow \mathfrak{X}^\bullet_G(M).
    \end{equation*}

    To see that this map intertwines the equivariant differentials let us take an arbitrary element $X \otimes p \in \mathfrak{X}_G^\bullet(\mathcal{F})$. Computation of the differential yields
    \begin{align*}
        \kappa \otimes \operatorname{id}_{\operatorname{Sym}^\bullet(\mathfrak{g}^*)} \big( \diff_{\Pi_{\mathcal{F}}, G} (X \otimes p) \big) &= \kappa \otimes \operatorname{id}_{\operatorname{Sym}^\bullet(\mathfrak{g}^*)} \bigg( \diff_{\Pi_{\mathcal{F}}} X \otimes p - \sum_{i = 1}^m \iota_{\lambda(X_i)} X \otimes \sigma^i p \bigg) \\
        &= i(\diff_{\Pi_{\mathcal{F}}} X) \otimes p - \sum_{i = 1}^m i(\iota_{\lambda(X_i)} X) \otimes \sigma^i p \\
        &= \diff_{\Pi} i(X) \otimes p - \sum_{i = 1}^m \iota_{s \lambda(X_i)} i(X) \otimes \sigma^i p \\
        &= \diff_{\Pi, G} (i(X) \otimes p) \\
        &= \diff_\Pi (\kappa \otimes \operatorname{id}_{\operatorname{Sym}^\bullet(\mathfrak{g}^*)}) (X \otimes p).
    \end{align*}
    In the third equality, the fact that $i(\iota_{\lambda} X) = \iota_{s \lambda} i(X)$ follows from the relation $\langle \lambda, X \rangle = \langle s \lambda, iX \rangle$, which can be proved from $\langle s \lambda, i X \rangle = \langle i^* s \lambda, X \rangle = \langle \lambda, X \rangle$. For the terms in the second summand, the computation is essentially the same together with the fact that $\diff_{\Pi, G} R = 0$. This proves the claim.
\end{proof}

\begin{corollary} \label{cor:equiv-formality-poisson}
    In the assumptions of \Cref{thm:equiv-poiss-splitting}, the equivariant Poisson cohomology groups are given by
    \begin{equation}
        \mathrm{H}_{\Pi, G}^\bullet(M) \cong \mathrm{H}_\Pi^\bullet(M) \otimes \operatorname{Sym}(\mathfrak{g}^*)^G.
    \end{equation}
    In other words, the complex $(\mathfrak{X}_G^\bullet(M), \diff_{\Pi, G})$ is equivariantly formal.
\end{corollary}

\begin{proof}
    The result is a direct combination of \Cref{thm:equiv-poiss-splitting}, \Cref{cor:equiv-formality-poisson}, and \Cref{thm:poisson-cohomology-cosymplectic}:
    \begin{align*}
        \mathrm{H}_{\Pi, G}^\bullet(M) &\cong \mathrm{H}_G^\bullet(\mathcal{F}) \oplus \mathrm{H}_G^{\bullet - 1}(\mathcal{F}) \\
        &\cong \mathrm{H}^\bullet(\mathcal{F}) \otimes \operatorname{Sym}^\bullet(\mathfrak{g}^*)^G \oplus \mathrm{H}^{\bullet - 1}(\mathcal{F}) \otimes \operatorname{Sym}^\bullet(\mathfrak{g}^*)^G \\
        &\cong \big( \mathrm{H}^\bullet(\mathcal{F}) \oplus \mathrm{H}^{\bullet - 1}(\mathcal{F}) \big) \otimes \operatorname{Sym}^\bullet(\mathfrak{g}^*)^G \\
        &\cong \mathrm{H}^\bullet_\Pi(M) \otimes \operatorname{Sym}^\bullet(\mathfrak{g}^*)^G. \qedhere 
    \end{align*}
\end{proof}

\bibliographystyle{alpha}
\bibliography{equivariant-cosymplectic-manifolds}

\end{document}